\documentclass[11pt]{amsart}

\usepackage{fullpage}
\usepackage{amsmath}
\usepackage{amssymb}
\usepackage{amsthm}
\usepackage{graphicx}
\usepackage{color}
\usepackage{mathdots}

\DeclareMathOperator*{\dist}{dist}

\DeclareMathOperator*{\inflow}{inflow}
\DeclareMathOperator*{\outflow}{outflow}

\newcommand{\Zz}{\mathbb{Z}}

\newtheorem{theorem}{Theorem}
\newtheorem{lemma}[theorem]{Lemma}

\newtheorem{corollary}[theorem]{Corollary}
\newtheorem{proposition}[theorem]{Proposition}

\newtheorem{definition}[theorem]{Definition}

\begin{document}

\title{Flow-firing processes}
\author{Pedro Felzenszwalb}
\address[Pedro Felzenszwalb]{Brown University}
\author{Caroline Klivans}
\address[Caroline Klivans]{Brown University}


\begin{abstract}

We consider a discrete non-deterministic \emph{flow-firing} process
for rerouting flow on the edges of a planar complex.  The process is
an instance of higher-dimensional chip-firing.  In the flow-firing
process, flow on the edges of a complex is repeatedly diverted across
the faces of the complex.  For non-conservative initial configurations
we show this process never terminates.  For conservative initial flows
we show the process terminates after a finite number of rerouting
steps, but there are many possible final configurations reachable from
a single initial state.  Finally, for conservative initial flows
around a topological hole we show the process terminates at a unique
final configuration.  In this case the process exhibits global
confluence despite not satisfying local confluence.
  
\end{abstract}

\maketitle

\section{Introduction}

We consider a discrete process for rerouting flow on the \emph{edges}
of a planar complex.  The process is a form of discrete diffusion; a
flow is repeatedly diverted according to a discrete Laplacian.  It is
also an instance of higher-dimensional chip-firing.  In the
flow-firing process considered here, flow is placed on the
$1$-dimensional cells of a complex and is rerouted across the
$2$-dimensional cells.  This is compared to graphical chip-firing,
where chips are placed on the vertices ($0$-dimensional cells) of a
graph and redistributed across the edges ($1$-dimensional cells).
Previous work on higher-dimensional chip-firing has considered
algebraic structures defined for finite complexes.  Here we consider
the dynamics of higher-dimensional chip-firing and work with infinite
complexes.

We focus on two important features of the flow-firing process --
whether or not the system is terminating and whether or not the system
is confluent.  To this end, three settings are explored.  We show
that:
\begin{itemize}
\item For non-conservative initial configurations, the process
  does not terminate (Section~\ref{sec:single}).
\item For conservative initial configurations, the
  process always terminates but does not have a unique terminating
  state.  The final configuration depends on the choices made during
  the firing process (Section~\ref{sec:terminate}).
\item For conservative initial configurations around a
  distinguished face (a topological hole), the process terminates in a
  unique state.  The final configuration is always the same regardless
  of the choices made during the firing process (Section~\ref{sec:confluent}).
\end{itemize}

See Figure~\ref{fig:summary} for an illustration of the three
different settings.

The third case is of particular interest.  The uniqueness of the final
configuration is an example of global confluence that does not follow
from local confluence, thus adding to an active narrative in
chip-firing, see Section~\ref{sec:flows}.

The special case of the $2$-dimensional grid is treated throughout
most of the paper for simplicity.  Section~\ref{sec:extensions}
discusses extensions to more general cases including arbitrary
planar graphs and higher dimensional polytopal decompositions.

\begin{figure}
\begin{tabular}{c}
  \parbox[c]{1.6in}{\includegraphics[height=1.3in]{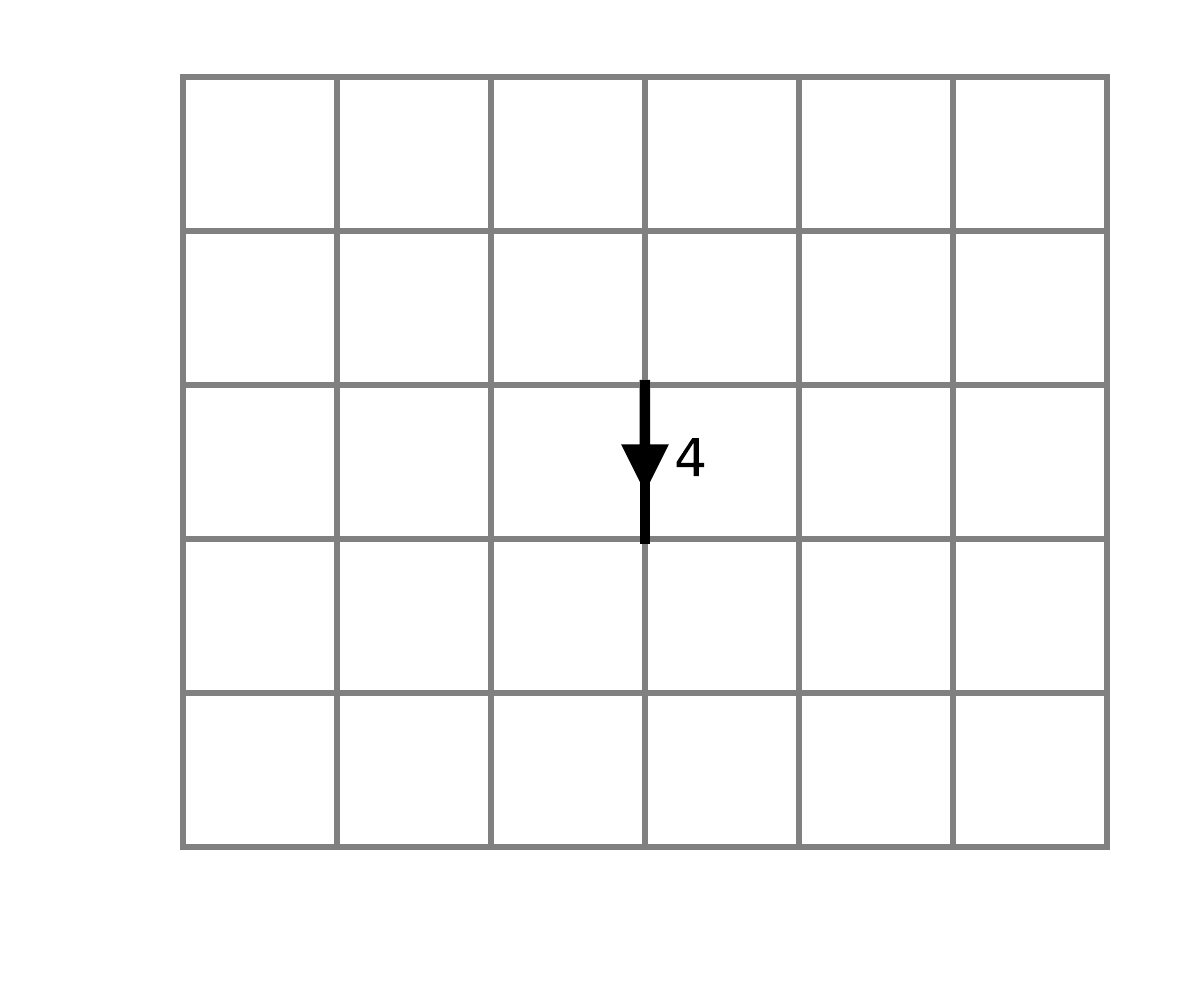}}
  {\Large $\rightarrow$}
  \parbox[c]{1.6in}{\includegraphics[clip=true,trim=20 15 0 0, height=1.3in]{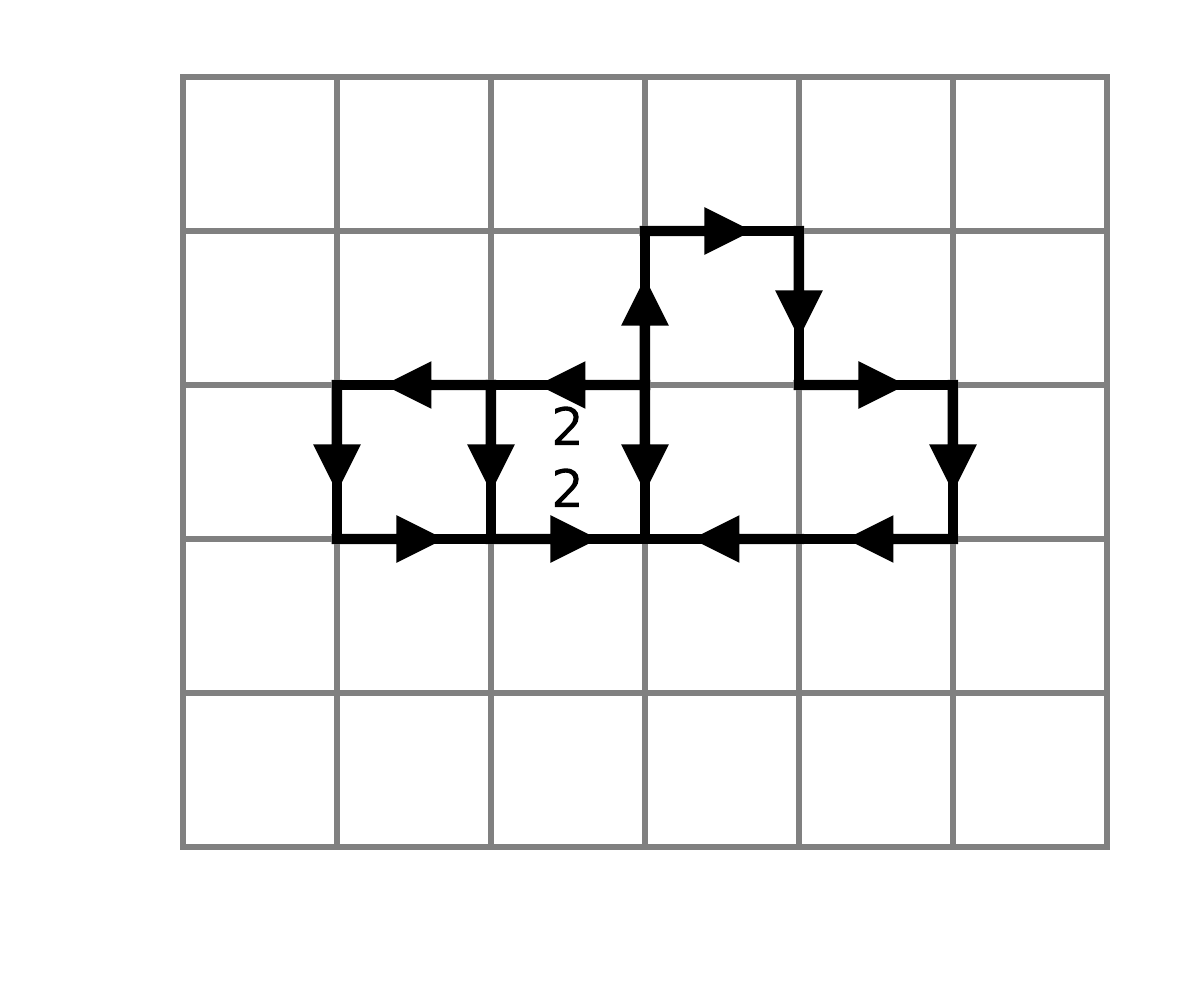}}
  {\Large $\rightarrow$}
  \parbox[c]{1.6in}{\includegraphics[clip=true,trim=20 15 0 0, height=1.3in]{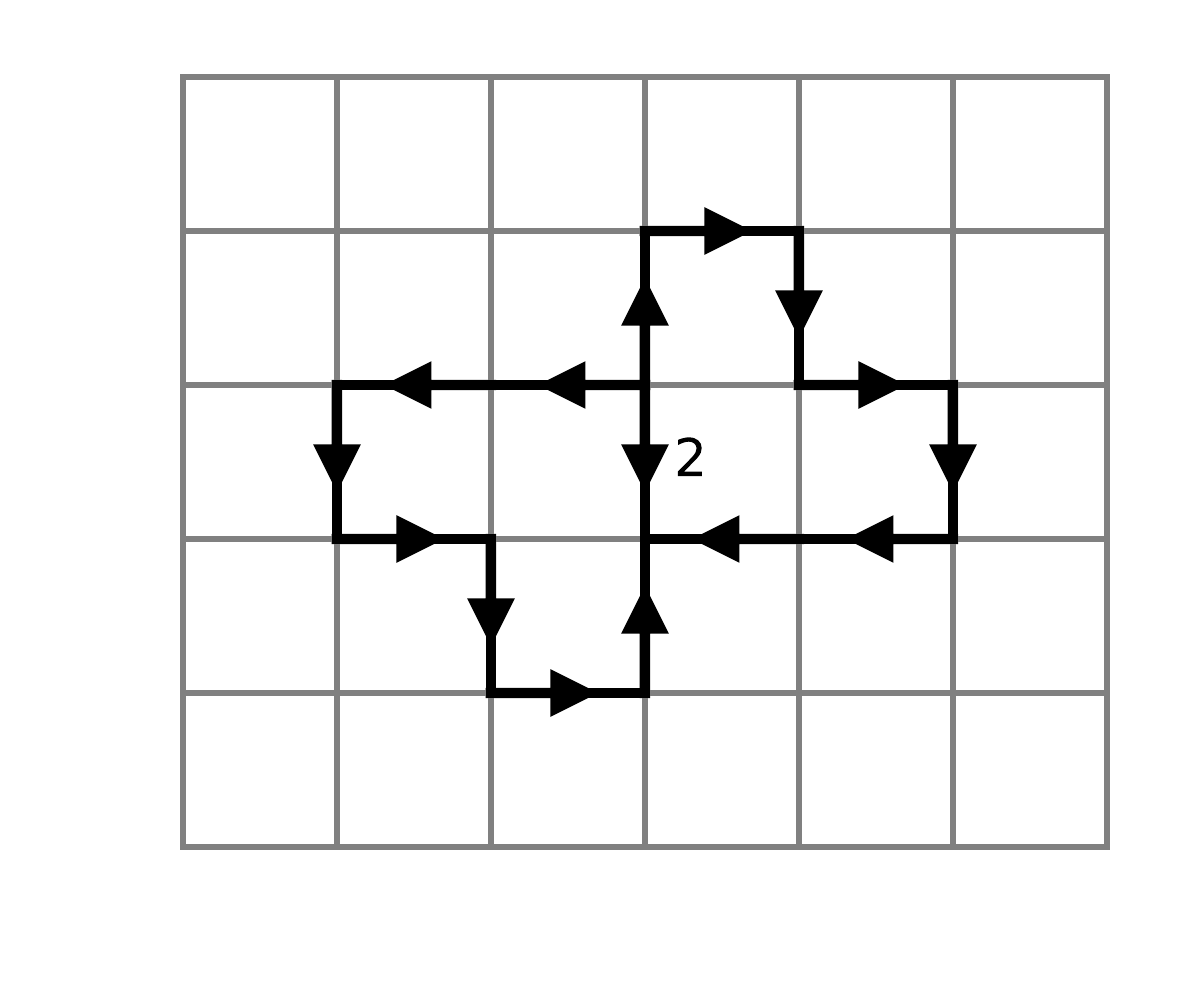}} {\Huge $\cdots$} \\
(a) Non-conservative flow: non-terminating process.
\end{tabular}

\vspace{.2cm}

\begin{tabular}{c}
  \parbox[c]{1.4in}{\includegraphics[clip=true,trim=15 15 0 0, height=1.3in]{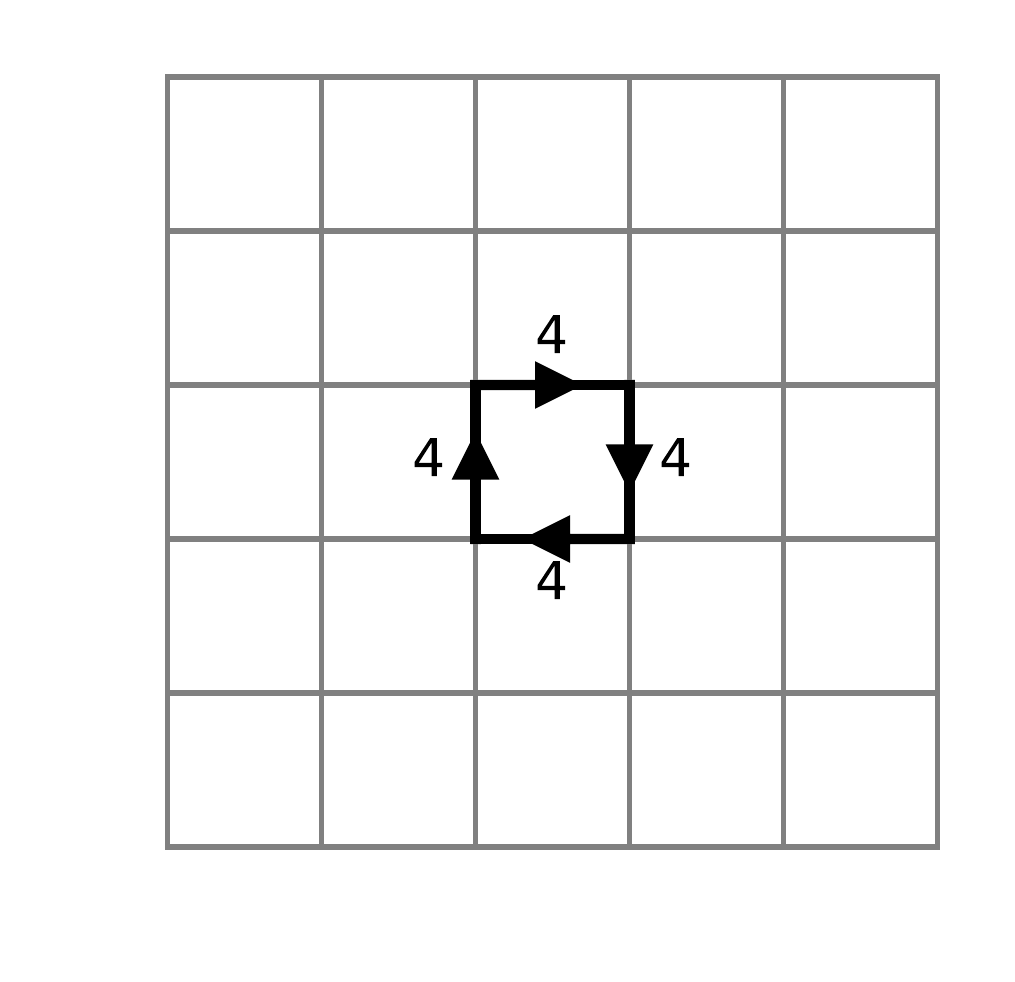}}
  {\Large \parbox[c]{2in}{\hbox{$\nearrow$} \hbox{$\searrow$}}}
  \parbox[c]{1.4in}{
    \hbox{\includegraphics[clip=true,trim=15 15 0 0, height=1.3in]{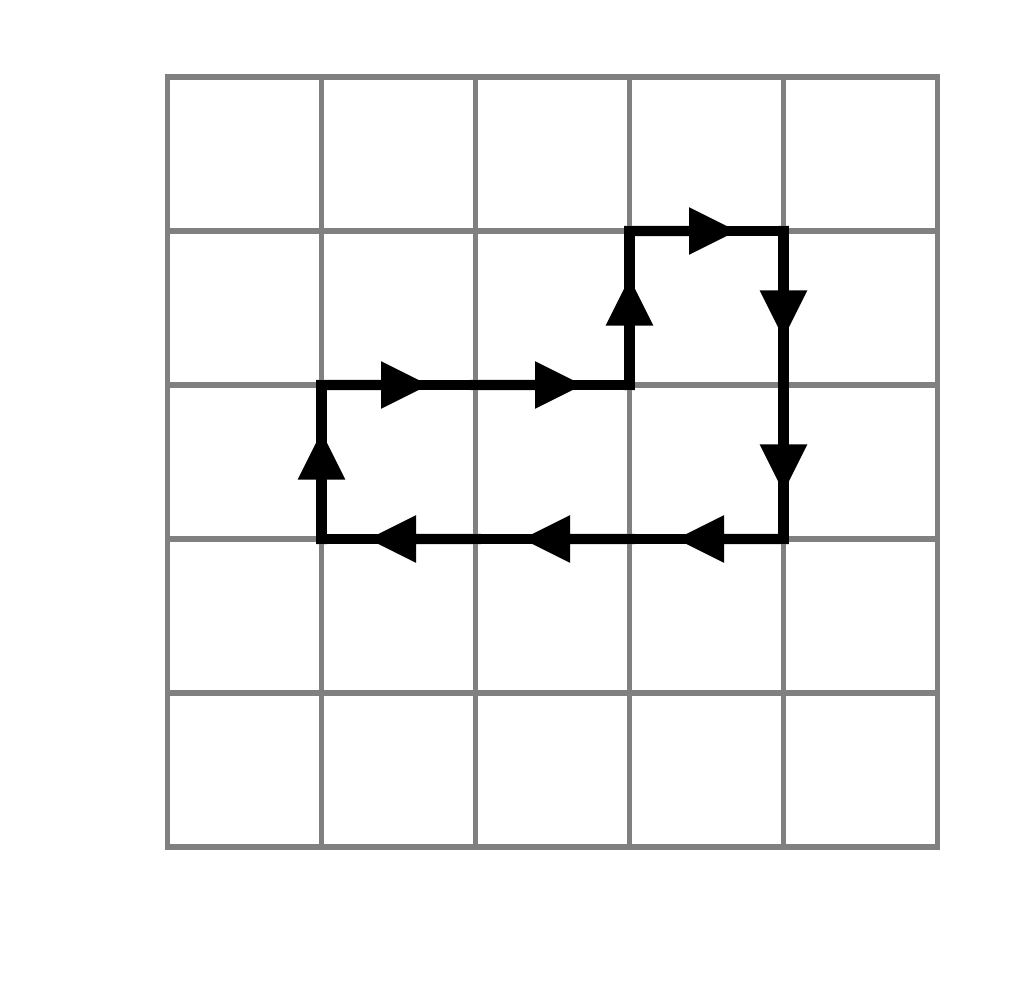}}
    \vspace{-.5cm}
    \hbox{\hspace{2cm}{\Huge $\vdots$}}
    \hbox{\includegraphics[clip=true,trim=15 15 0 0, height=1.3in]{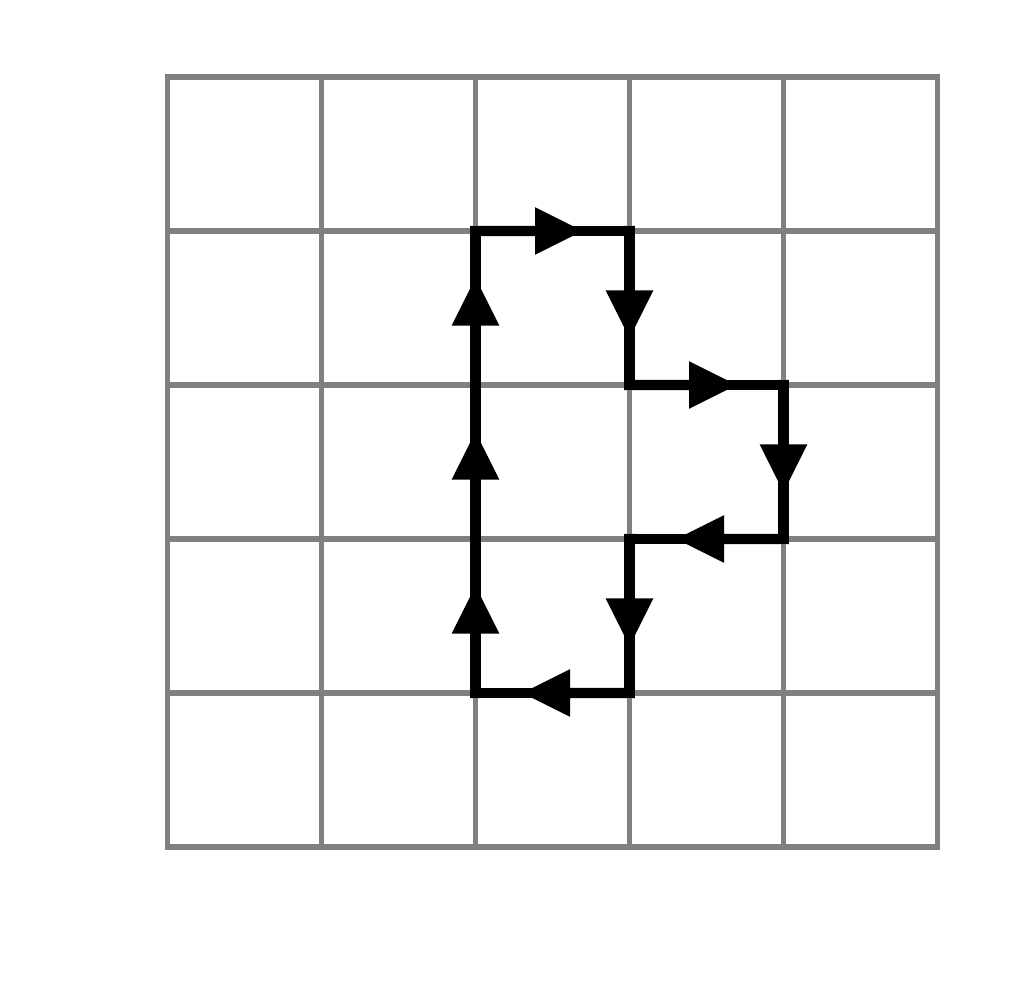}}} \\
  (b) Circulation: terminating but non-unique final configuration.
\end{tabular} 

\vspace{.2cm}

\begin{tabular}{c}
  \parbox[c]{2.9in}{\includegraphics[clip=true,trim=30 30 0 0, height=2.85in]{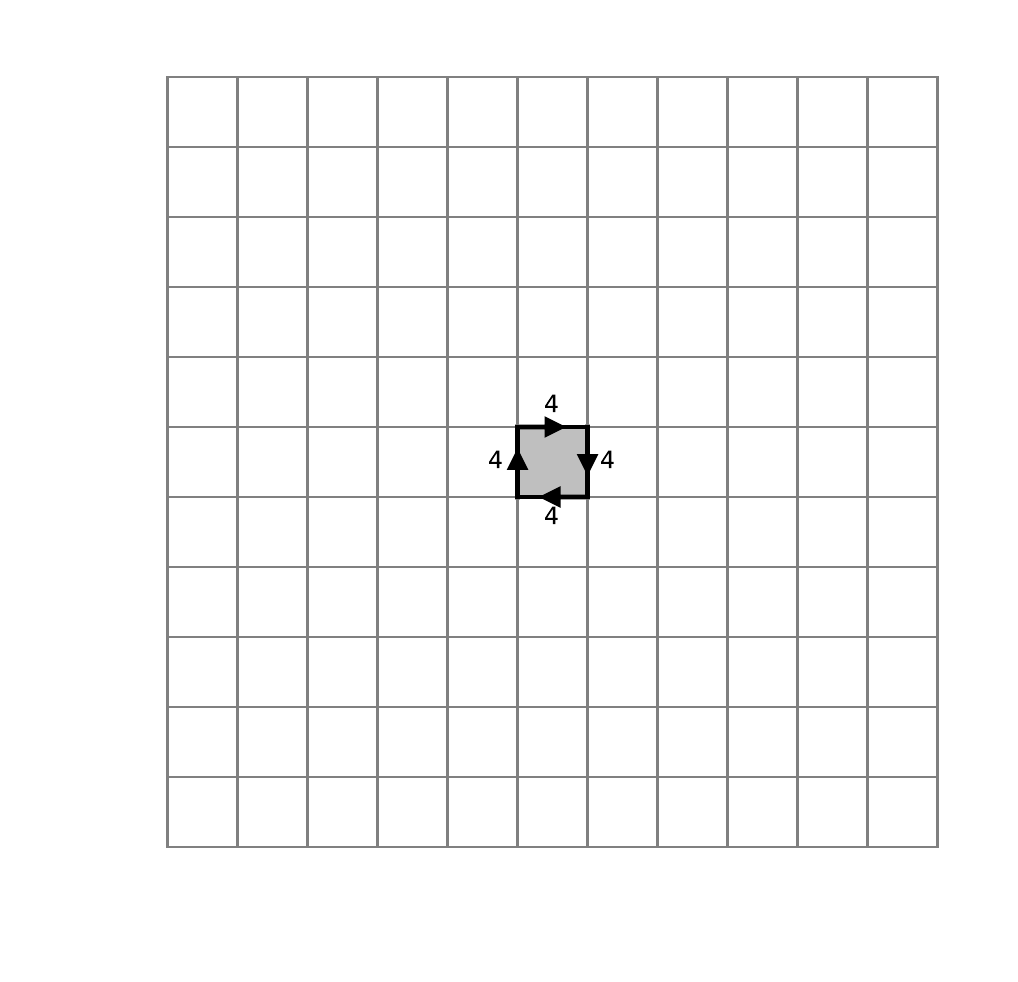}}
  {\Large $\rightarrow$}
  \parbox[c]{2.9in}{\includegraphics[clip=true,trim=30 30 0 0, height=2.85in]{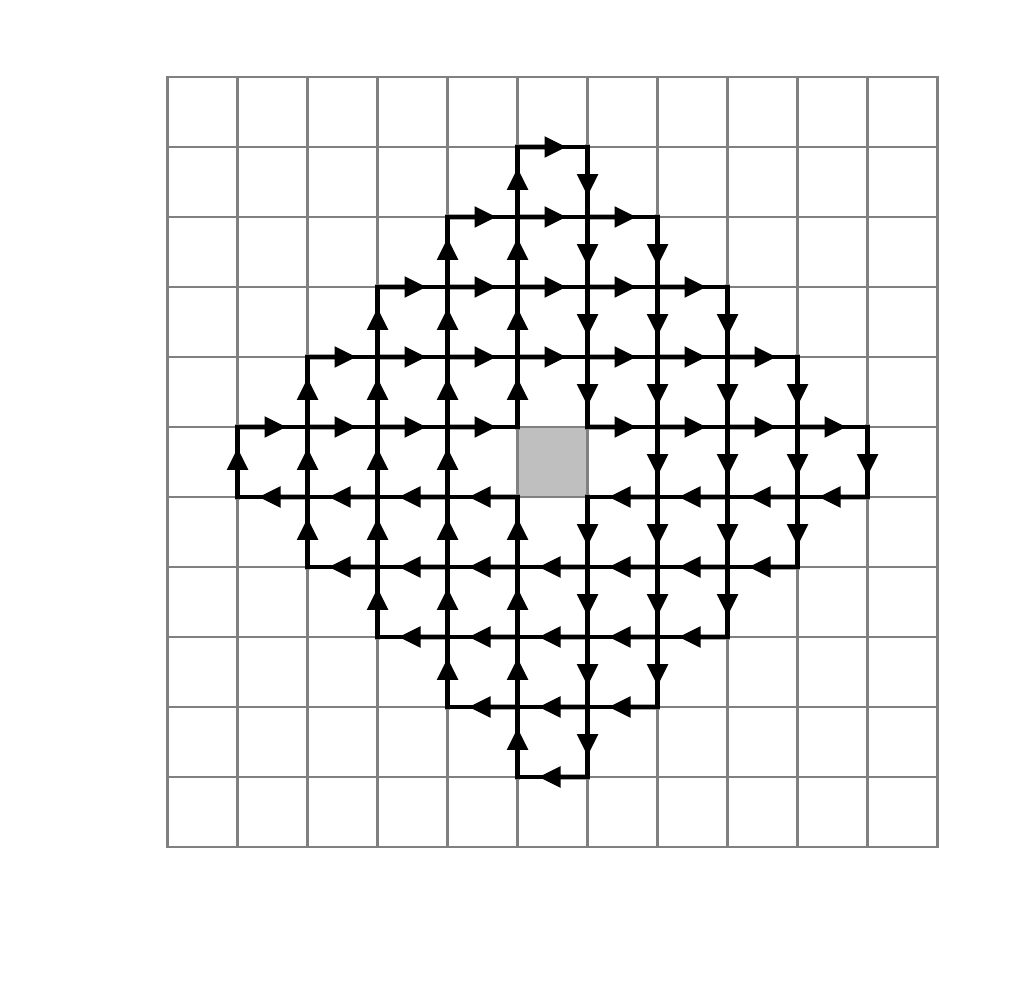}} \\
  (c) Circulation around a hole: terminating and unique final configuration
\end{tabular}
\caption{Flow-firing in three different settings.}
\label{fig:summary}
\end{figure}

\section{Flow-firing}
\label{sec:flows}

Let $G$ be the (infinite) grid graph embedded as $\mathbb{Z}^2$.  For
bookkeeping purposes, we orient each edge from South to North and West
to East.  The flow-firing process on $G$ involves configurations of
integral flow on the edges of the graph.

\begin{definition} A \emph{flow configuration} for $G$ is an integer assignment 
  $f$ specifying an amount of flow on each edge.  Negative values
  signify that the flow is oriented opposite that of the edge itself.
\end{definition}

Figure~\ref{fig:config} illustrates a flow configuration and the
corresponding integer vector.\footnote{In our terminology a flow may
  or may not be conservative at each vertex.  Other sources reserve
  the name ``flow'' for the more restricted case.}

Let $e$ be an edge and $\sigma$ a face (square) that contains $e$.
\emph{Rerouting} a unit of flow on $e$ across $\sigma$ replaces one
unit of flow along $e$ with one unit of flow along the alternate path
formed by the other edges of $\sigma$, see Figure~\ref{fig:reroute}.

\begin{figure}
  \centerline{\parbox[c]{3.5cm}{\includegraphics[width=3cm]{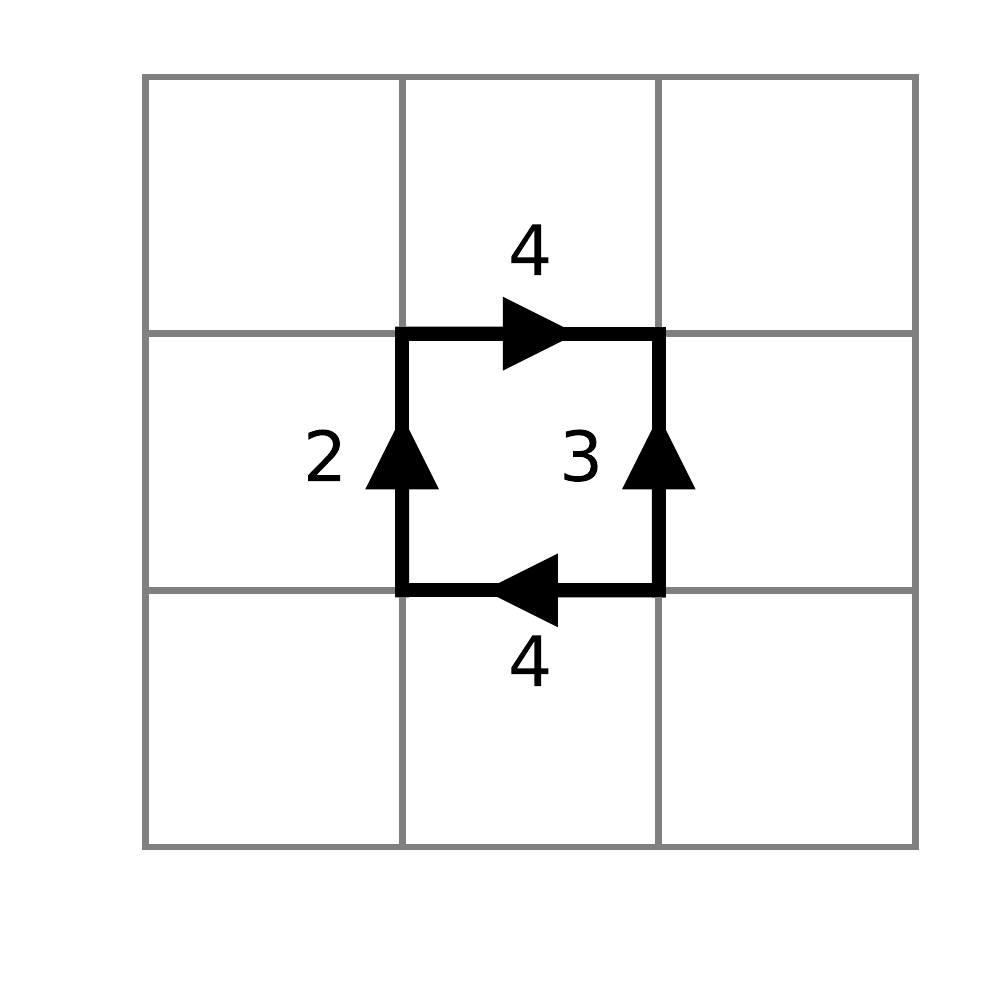}}
    $f=(\ldots,2,3,\ldots,-4,4,\ldots)$ }
  \caption{A flow configuration.}
  \label{fig:config}
\end{figure}

\begin{figure}
  \centerline{\includegraphics[width=1.2in]{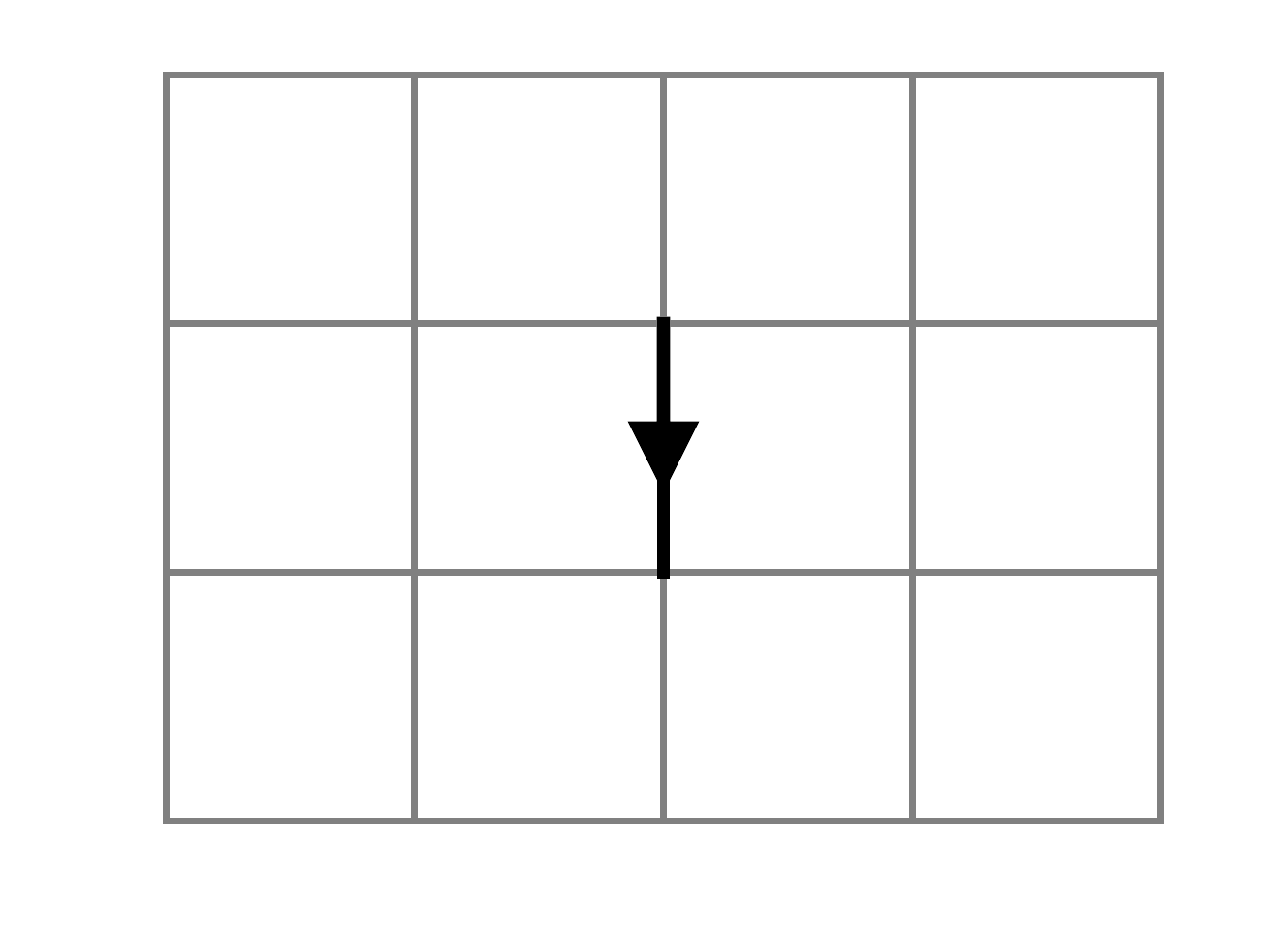}}
  \centerline{\parbox[c]{1.25in}{\includegraphics[width=1.2in]{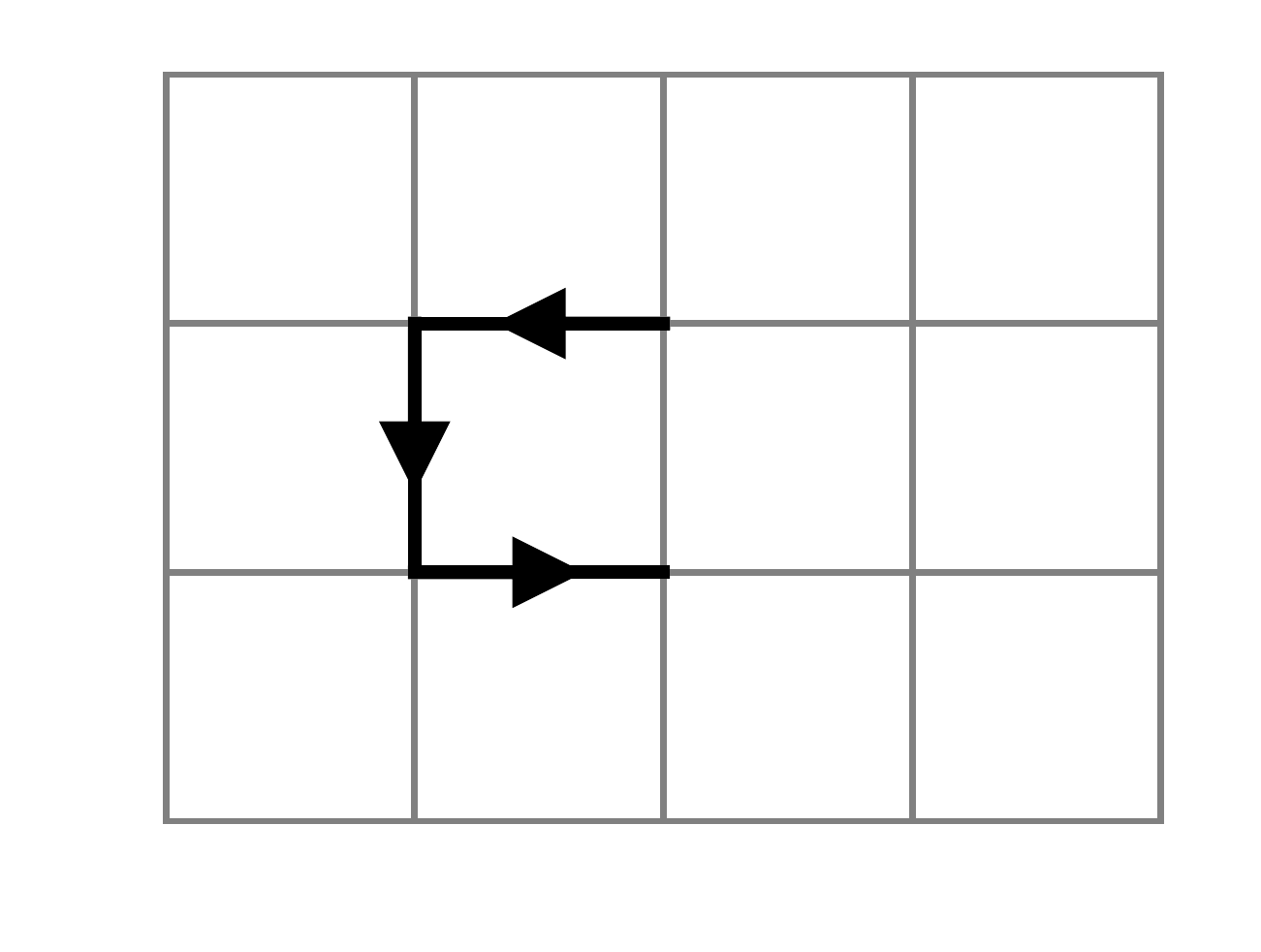}} or 
    \parbox[c]{1.25in}{\includegraphics[width=1.2in]{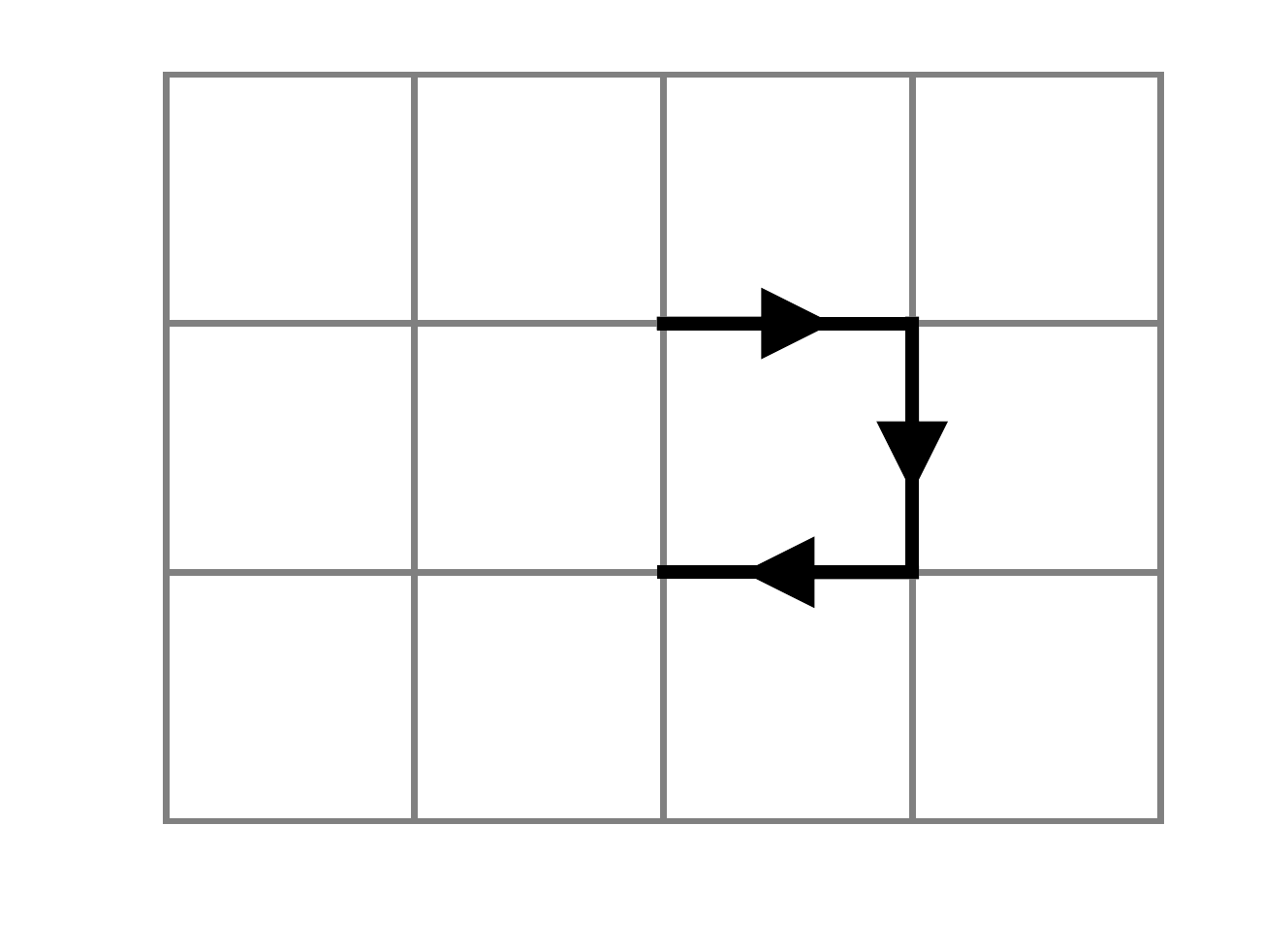}}}
  \caption{Rerouting a unit of flow.  A unit of flow along an edge, as
    in the top, can reroute across a face to the left or to the right,
    resulting in one of the two configurations on the bottom.}
  \label{fig:reroute}
\end{figure}

If an edge has two units of flow (in either direction) we can reroute
one unit around each of the two faces containing $e$.  We are now
ready to define the flow-firing process.\\

\fbox{
  \parbox{.9\textwidth}{
    $\,$ \\
    \noindent{\bf The flow-firing process}\\
    For the grid graph\\
    $\,$ \\
    At each step:  
    \begin{itemize}  
    \item Choose an edge $e$ with $2$ or more units of flow (in either
      direction).
    \item Fire $e$ by rerouting 1 unit of flow around each of the
      two faces containing $e$. 
    \end{itemize}
    \vspace{.2cm}
}} \\

Figure~\ref{fig:process} shows the flow-firing process on an initial
configuration consisting of $2$ units of flow on a single edge.
Figure~\ref{fig:firing} shows an example of the flow-firing process
from a larger initial configuration.  

\begin{figure}
  \centering
  \parbox[c]{1.3in}{\includegraphics[width=1.2in]{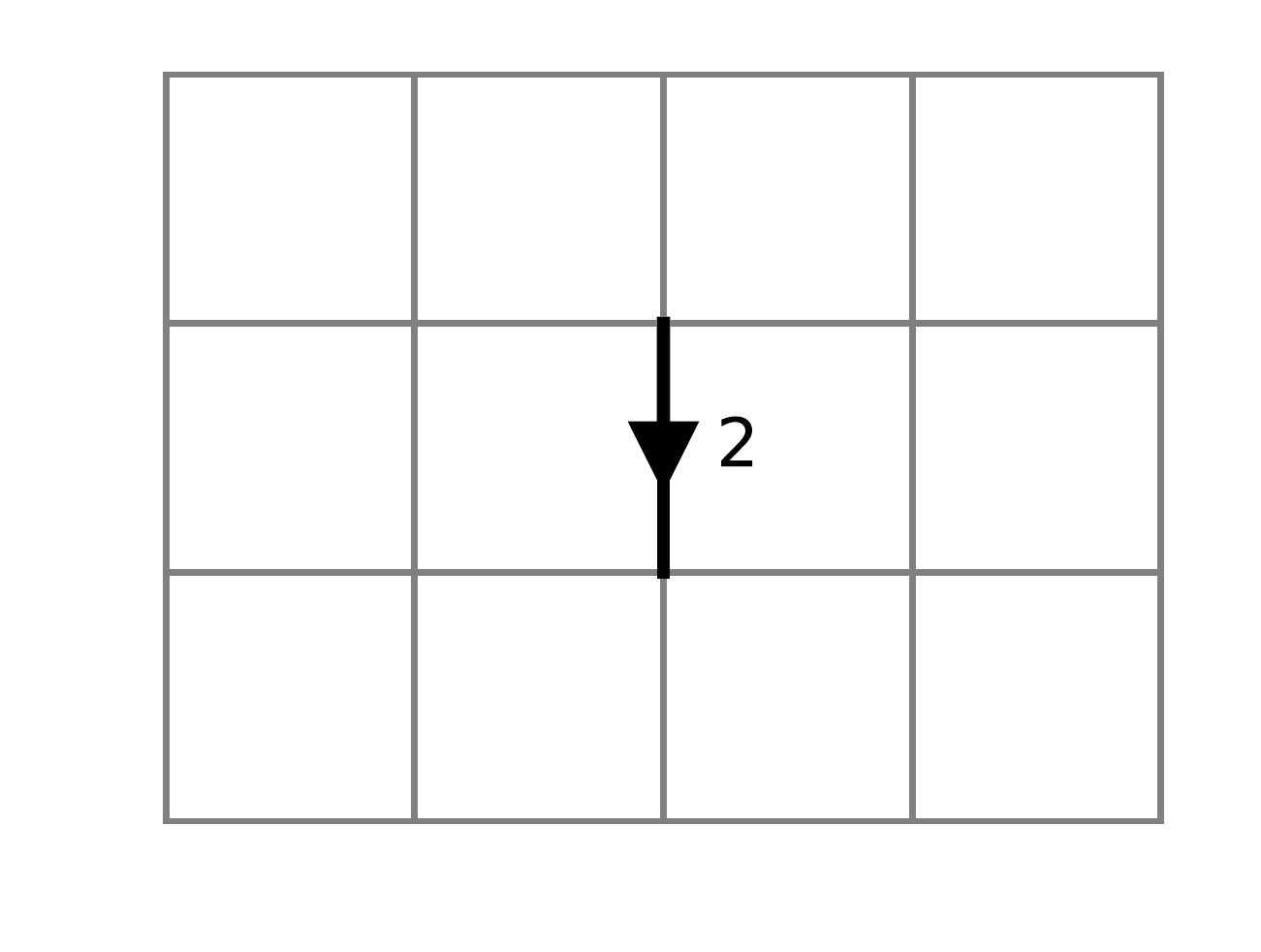}}
  \Large $\rightarrow$
  \parbox[c]{1.3in}{\includegraphics[width=1.2in]{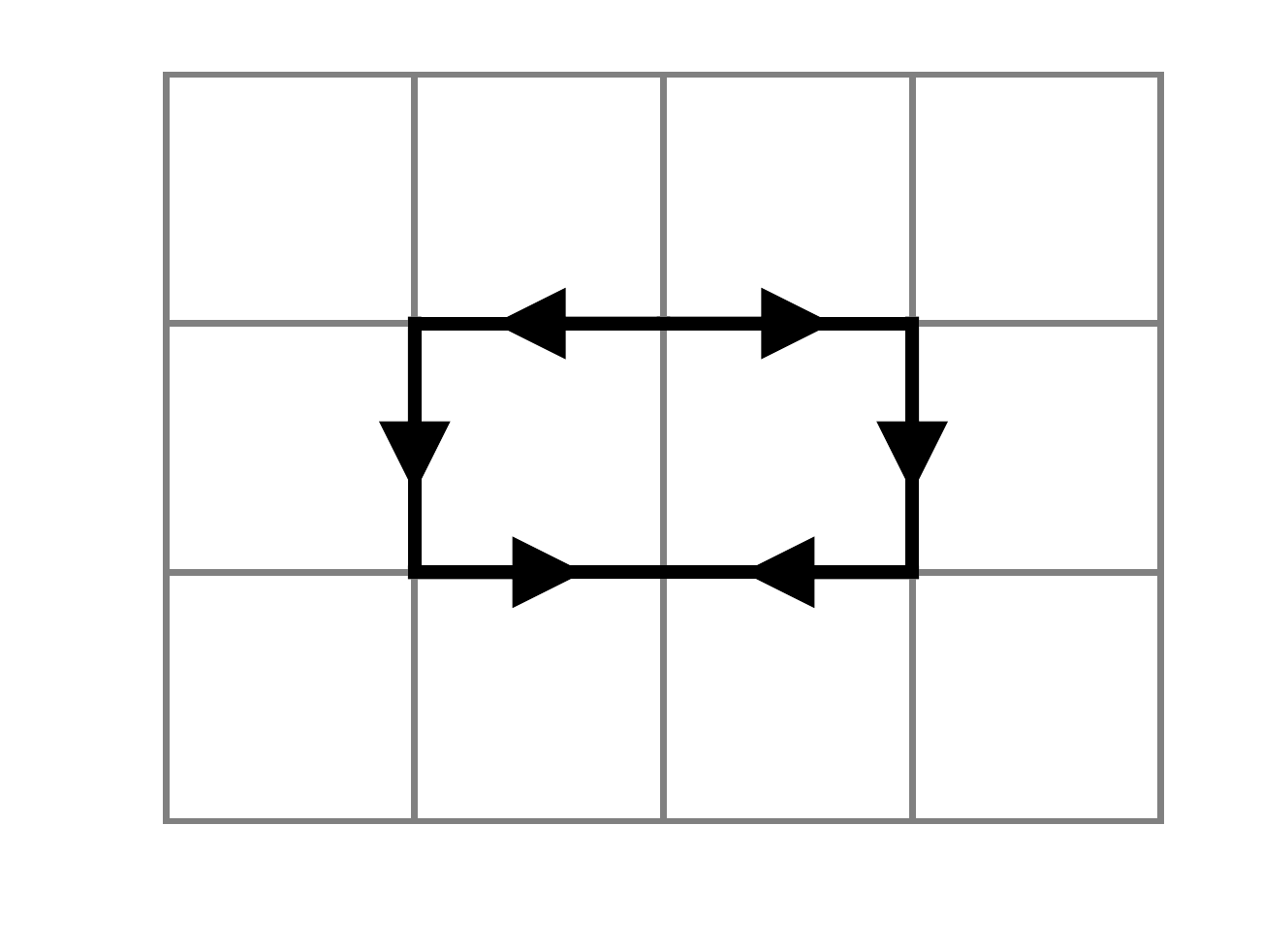}}
  
  \parbox[c]{1.3in}{\includegraphics[width=1.2in]{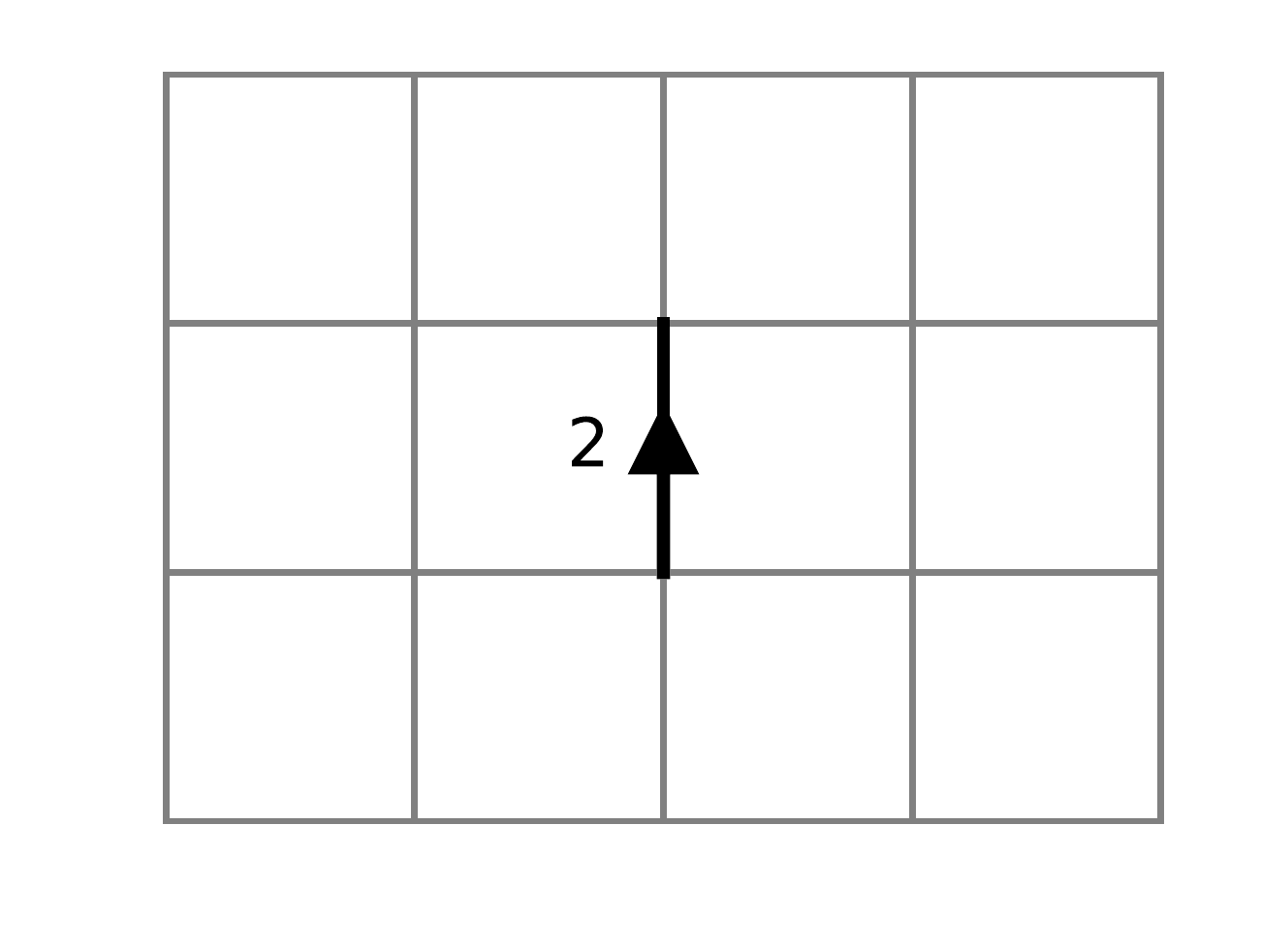}}
  \Large $\rightarrow$
  \parbox[c]{1.3in}{\includegraphics[width=1.2in]{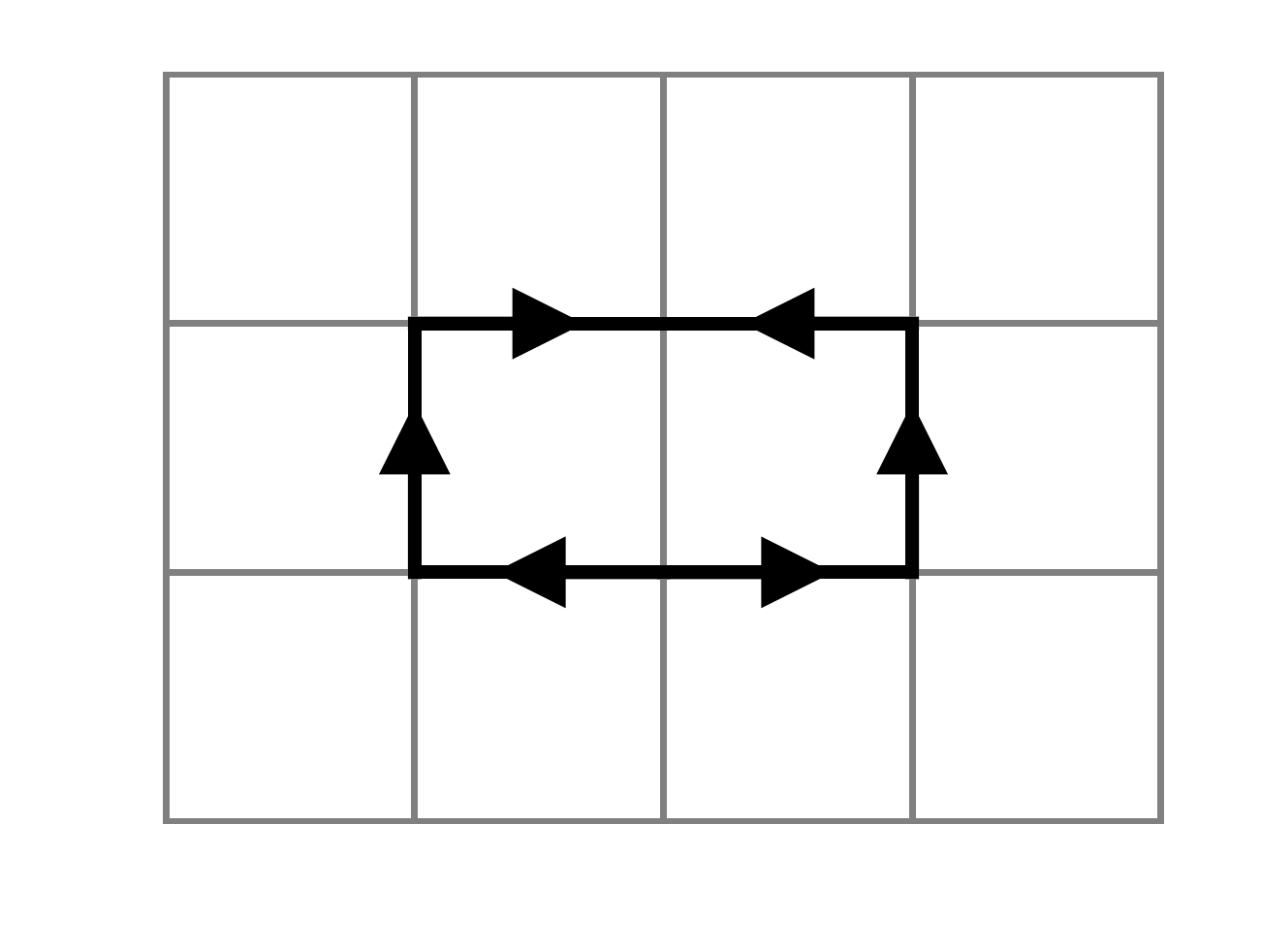}}
  \caption{The flow-firing process.  An edge can fire when it has at
    least two units of flow in either direction.}
  \label{fig:process}
\end{figure}

\begin{figure}
 \centering
  \parbox[c]{1.3in}{\includegraphics[width=1.2in]{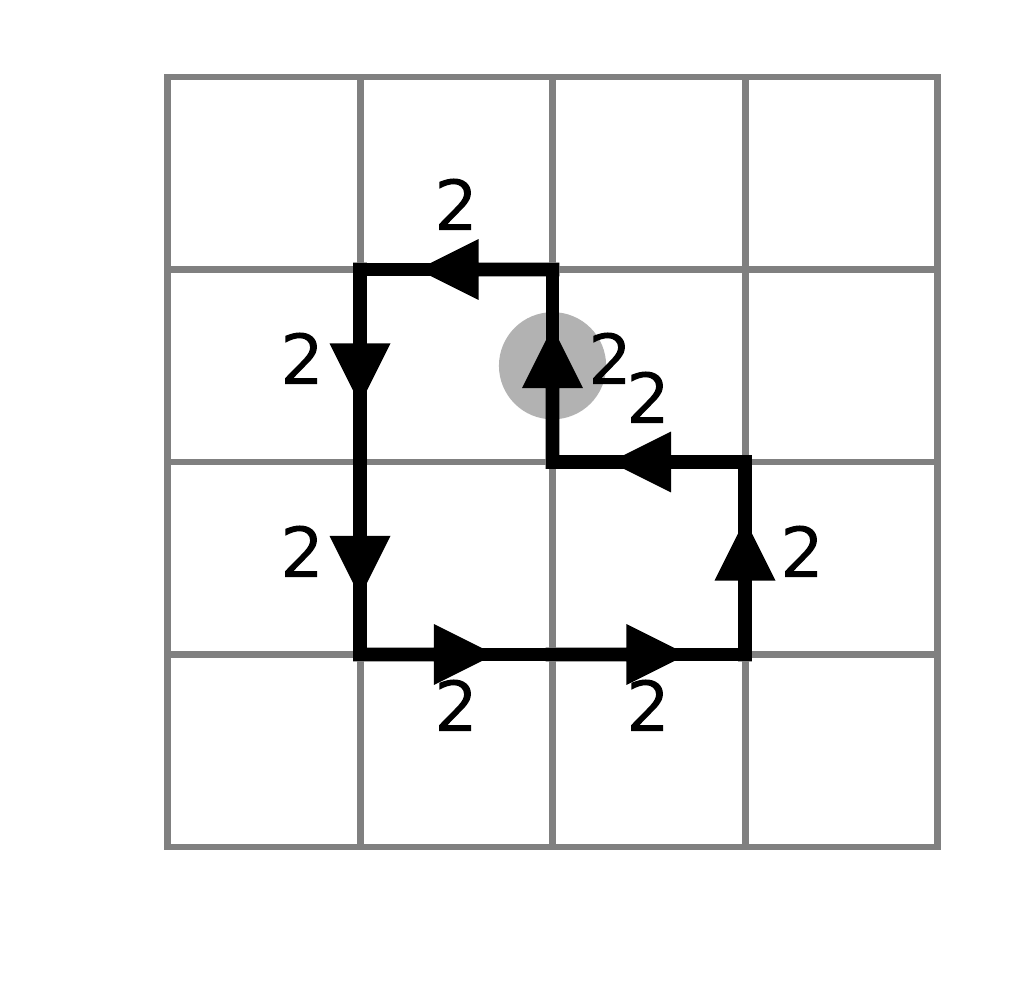}}
  \Large $\rightarrow$
  \parbox[c]{1.3in}{\includegraphics[width=1.2in]{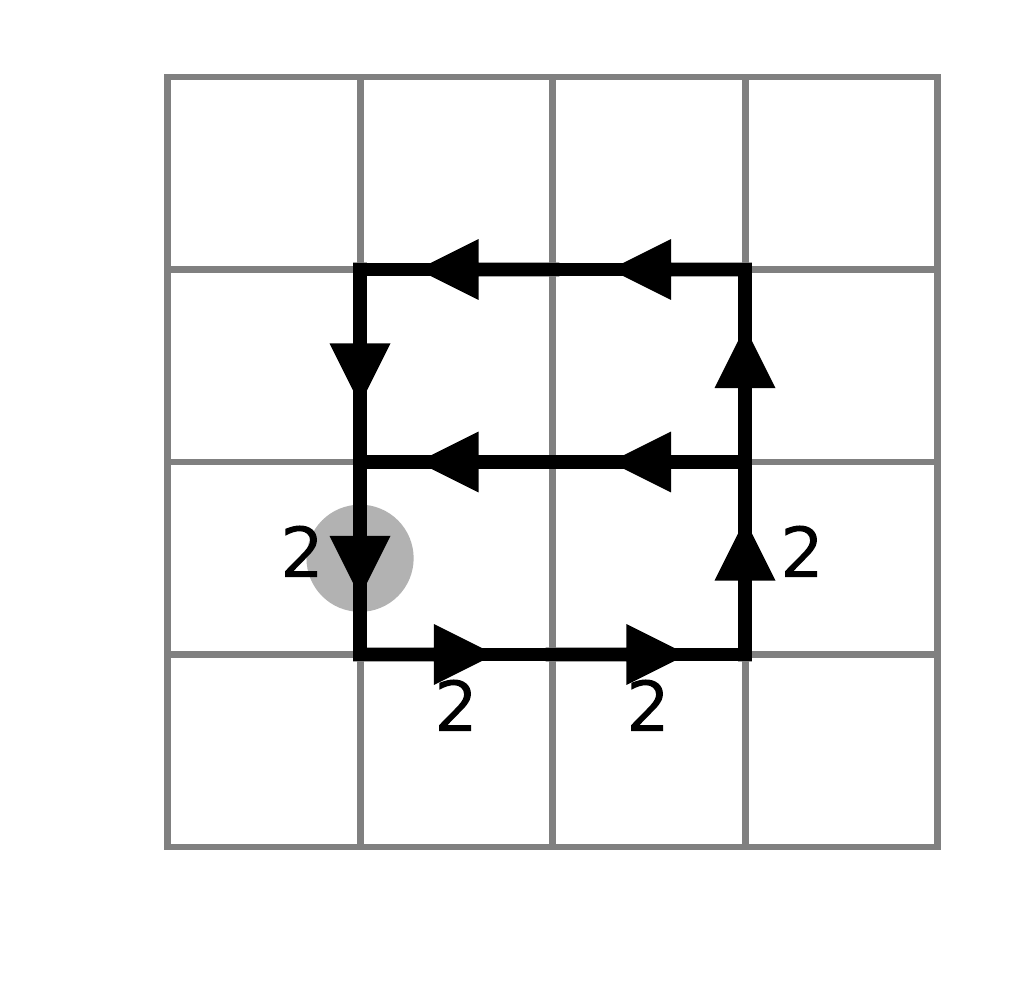}}
  \Large $\rightarrow$
  \parbox[c]{1.3in}{\includegraphics[width=1.2in]{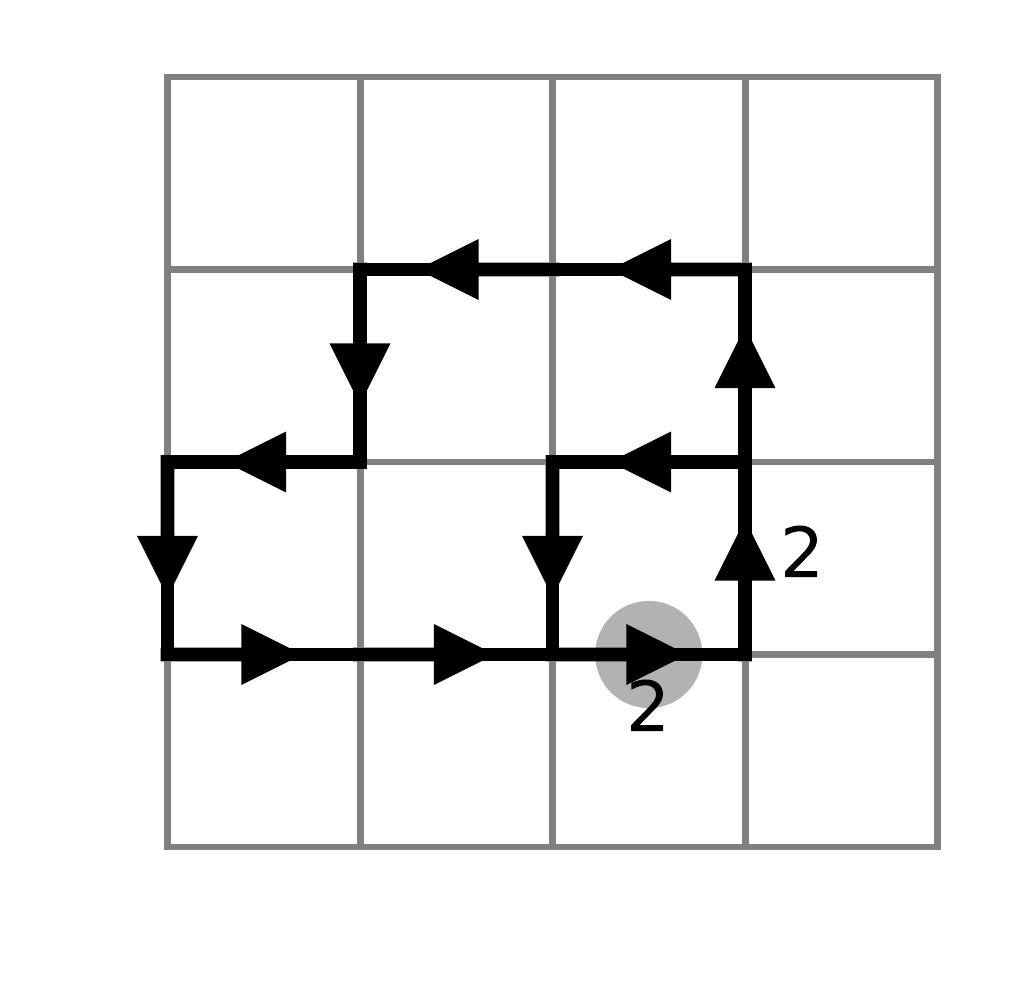}}
  \Large $\rightarrow$
  \parbox[c]{1.3in}{\includegraphics[width=1.2in]{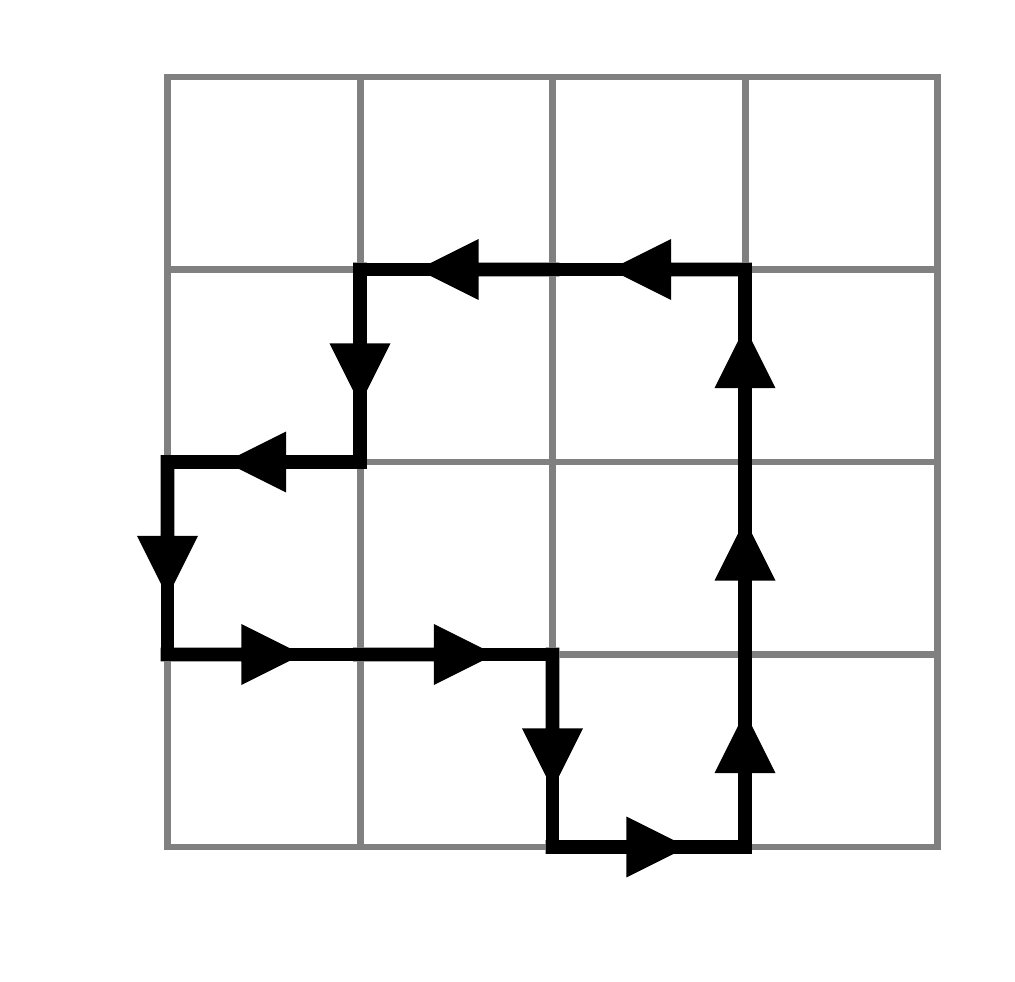}}
\caption{Example of the flow-firing process.  In each step the
  highlighted edge \emph{fires} and 2 units of flow are rerouted
  across two faces.  The process terminates when no edge
  has 2 or more units of flow.}
\label{fig:firing}
  \end{figure}

The flow-firing process is a form of higher-dimensional chip-firing as
introduced in~\cite{Critical}, see also~\cite{Book}[Chapter 7].  In
graphical chip-firing, a \emph{chip configuration} is an integer
assignment (of chips) to the \emph{vertices} of a graph.  The firing
rule is that a vertex $v$ can fire if the number of chips at $v$ is at
least $\deg(v)$.  Firing $v$ decreases the value at $v$ by $\deg(v)$
and increases the value at each neighbor of $v$ by $1$, see
Figure~\ref{fig:pgraph}.

\begin{figure}
  \centering
  \parbox[c]{1.3in}{\includegraphics[width=1.2in]{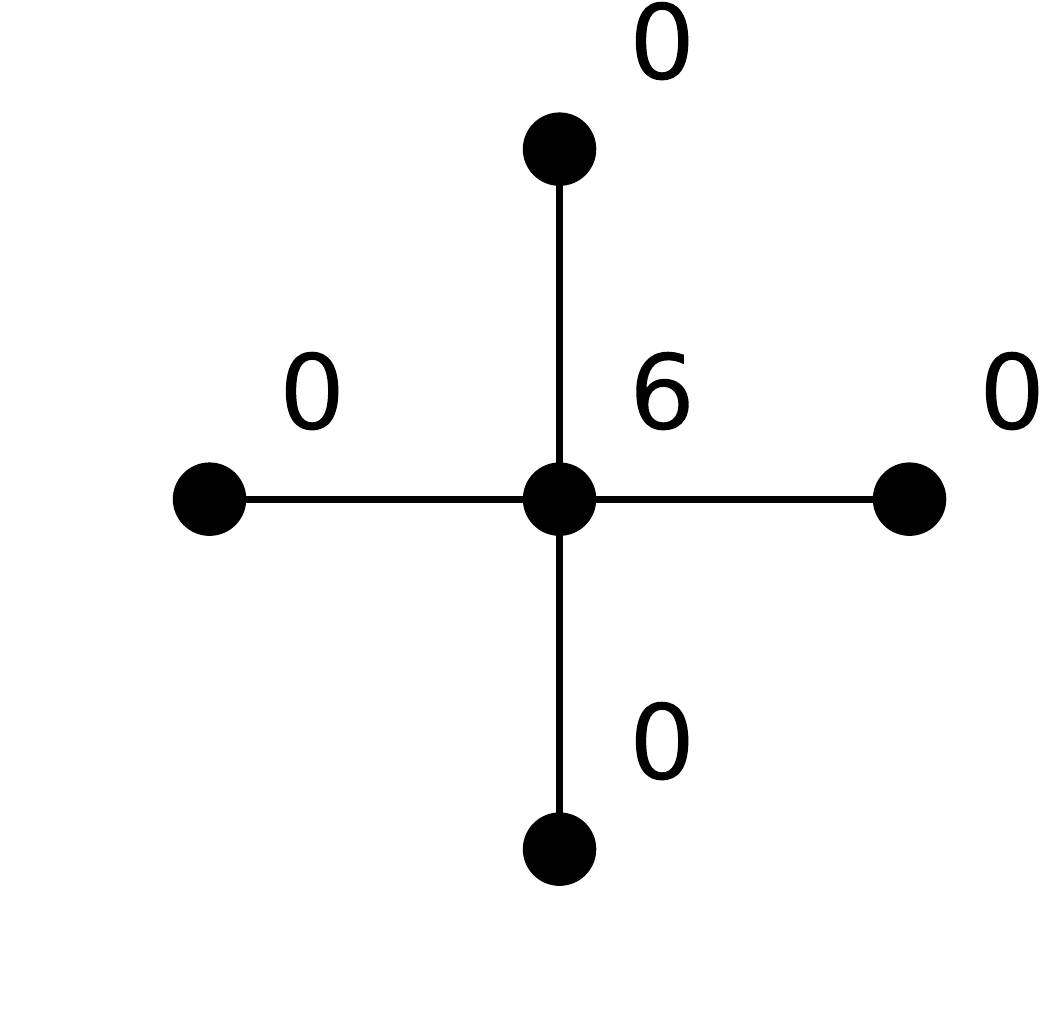}}
  \Large $\rightarrow$
  \parbox[c]{1.3in}{\includegraphics[width=1.2in]{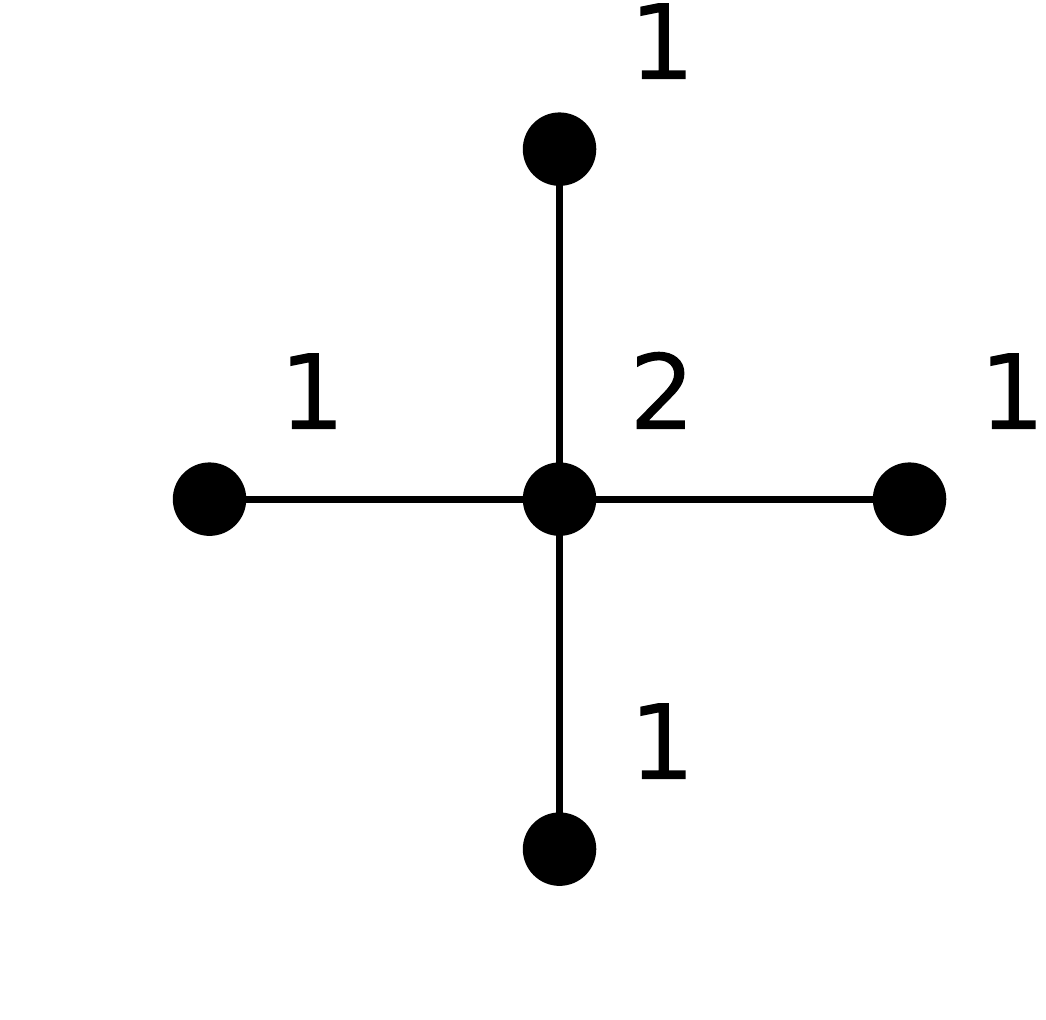}}
  \caption{Graphical chip-firing.}
  \label{fig:pgraph}
\end{figure}

Graphical chip-firing is $1$-dimensional.  Chips are on
$0$-dimensional vertices and move across $1$-dimensional edges.
Flow-firing is $2$-dimensional.  Flow is on $1$-dimensional edges and
moves across $2$-dimensional faces.  The degree of an edge $e$ is the
number of faces containing $e$.  For the grid graph, the degree of
each edge is $2$, hence the need for $2$ units of flow for an edge to
fire.  Two edges are neighbors if they are contained in a common face.
When an edge $e$ fires, flow is diverted from $e$ to neighboring
edges.

The result of firing a vertex in graphical chip-firing can be
expressed in terms of the graph Laplacian, $\Delta_1(G)$.  If $c'$ is
the configuration obtained from $c$ after firing vertex $i$, then $c'
= c - \Delta_1(G)e_i$.  Similarly, flow-firing can be expressed in
terms of a combinatorial Laplacian, $\Delta_2(G)$.  The
two-dimensional Laplacian of a complex reflects the degrees and
incidence relations between faces of the complex.  If $f'$ is the
configuration obtained from $f$ after firing edge $i$, then $f' = f
\pm \Delta_2(G)e_i$.  The sign of the update depends on the
orientation of the flow on edge $i$ in $f$.

There are important differences between $1$-dimensional and
$2$-dimensional chip-firing.  Graphical chip-firing on the infinite
grid always terminates if started from a chip configuration with
finite support.  This is not the case in flow-firing.  In
Section~\ref{sec:terminate} we show that the flow-firing process
always terminates if started from a \emph{conservative} flow (a
circulation) with finite support.

Another important difference is that in graphical chip-firing, the
total number of chips is conserved.  In the flow-firing process, the
quantity $$\inflow(v) - \outflow(v)$$ is conserved (at each vertex)
instead.  

Note that in the flow-firing process a rerouting operation can lead to
cancellation when flow runs in opposite directions across an
edge, see Figure~\ref{fig:firing}.

The cancellation of flow, as seen in Figure~\ref{fig:firing}, cannot
happen in graphical chip-firing.  When a vertex $v$ fires, the number
of chips at vertices other than $v$ can only increase.  This simple
observation leads to an important property of graphical chip-firing
known as local confluence.  Local confluence refers to the fact that
from a fixed configuration $c$, if two different states $c_1$ and
$c_2$ can be reached after a single step, then there is a common state
reachable from both $c_1$ and $c_2$ after a single step.  A system
that satisfies this property is also said to satisfy the diamond
lemma.  Local confluence in terminating systems implies \emph{global
  confluence} \cite{New}.  In graphical chip-firing, this fact
tells us that, if the chip-firing process from an initial
configuration terminates, then it terminates in a unique final
configuration regardless of the choices made at each step.

Because of cancellation of flow, the flow-firing process does not
satisfy the diamond lemma, see Figure~\ref{fig:diamondfail}.  In fact,
we show that in the flow-firing process there can be many terminating
states from a single initial configuration, see
Section~\ref{sec:terminate}.

In Section~\ref{sec:confluent} we consider a modification of
flow-firing which displays global confluence despite not satisfying
local confluence.  Global confluence without local confluence has
recently been observed in other chip-firing contexts.  \emph{Labeled
  chip-firing}, see~\cite{Sorting}, is an example. Labeled chips are
fired along the path graph, larger to the right and smaller to the
left.  Depending on the parity of the initial number of chips, the
labeled chip-firing process terminates in a unique final configuration
(with the chips sorted) even though the process does not satisfy the
diamond lemma.

The chip-firing processes on root systems introduced in
\cite{galashin2017root} and \cite{galashin2} generalize labeled
chip-firing.  Again, the root systems processes do not satisfy the
diamond lemma, nonetheless many cases do display global confluence.

In general, establishing global confluence without local confluence
has proved difficult.  For the flow-firing process we introduce a
topological hole to the grid (remove a square) and show that starting
from a conservative flow around the hole, the flow-firing process
satisfies global confluence despite not satisfying the diamond lemma,
see Theorem~\ref{thm:main}.
 
\begin{figure}
  \centerline{\includegraphics[width=1.2in]{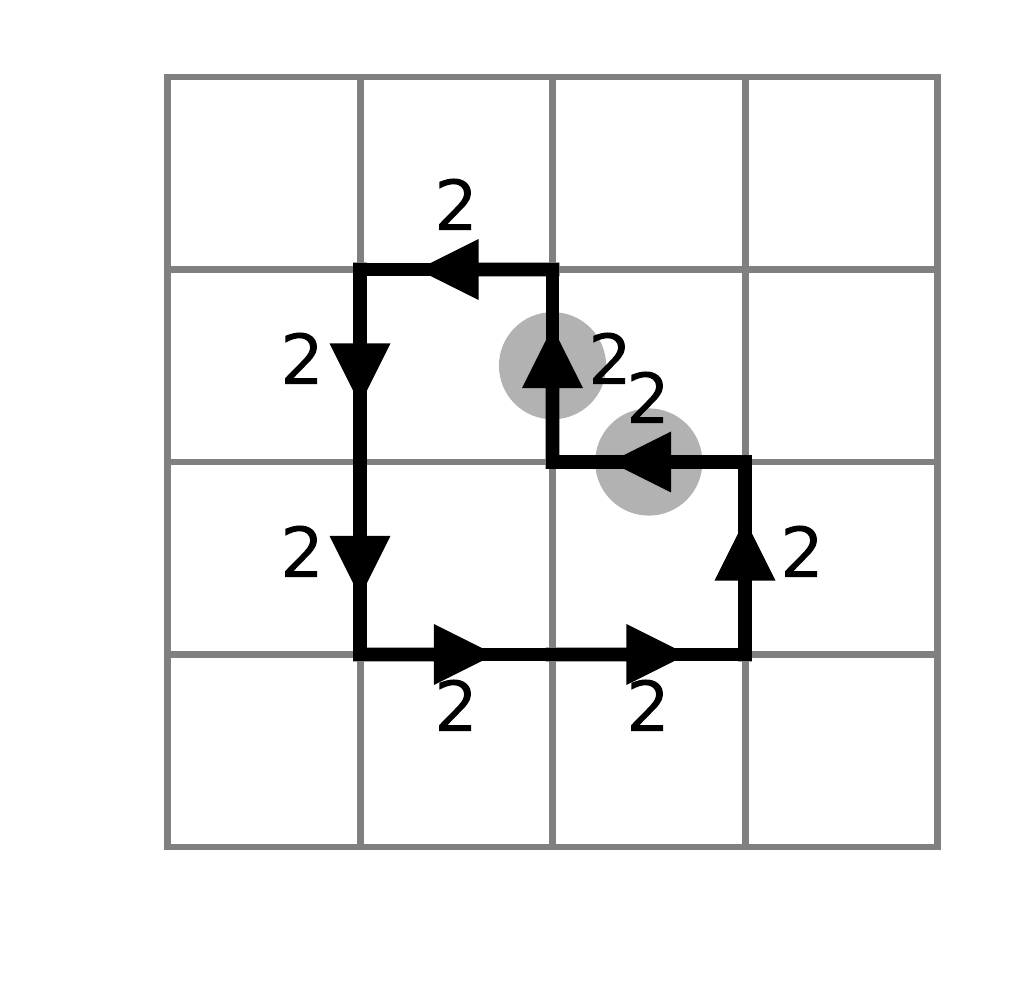}}
  \vspace{-.3cm} 
  \centerline{\Large $\;$ $\swarrow$ $\;\;\;$ $\searrow$}
  \centerline{\includegraphics[width=1.2in]{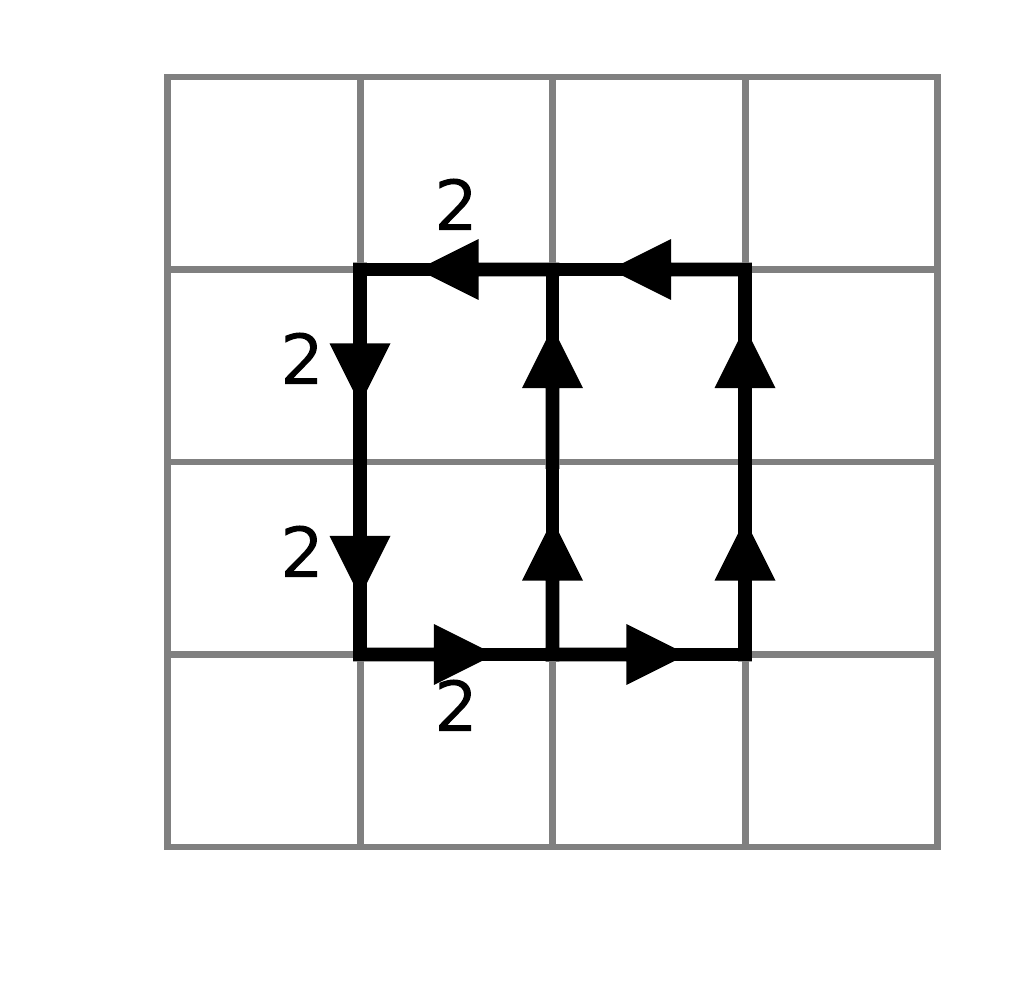}
    \includegraphics[width=1.2in]{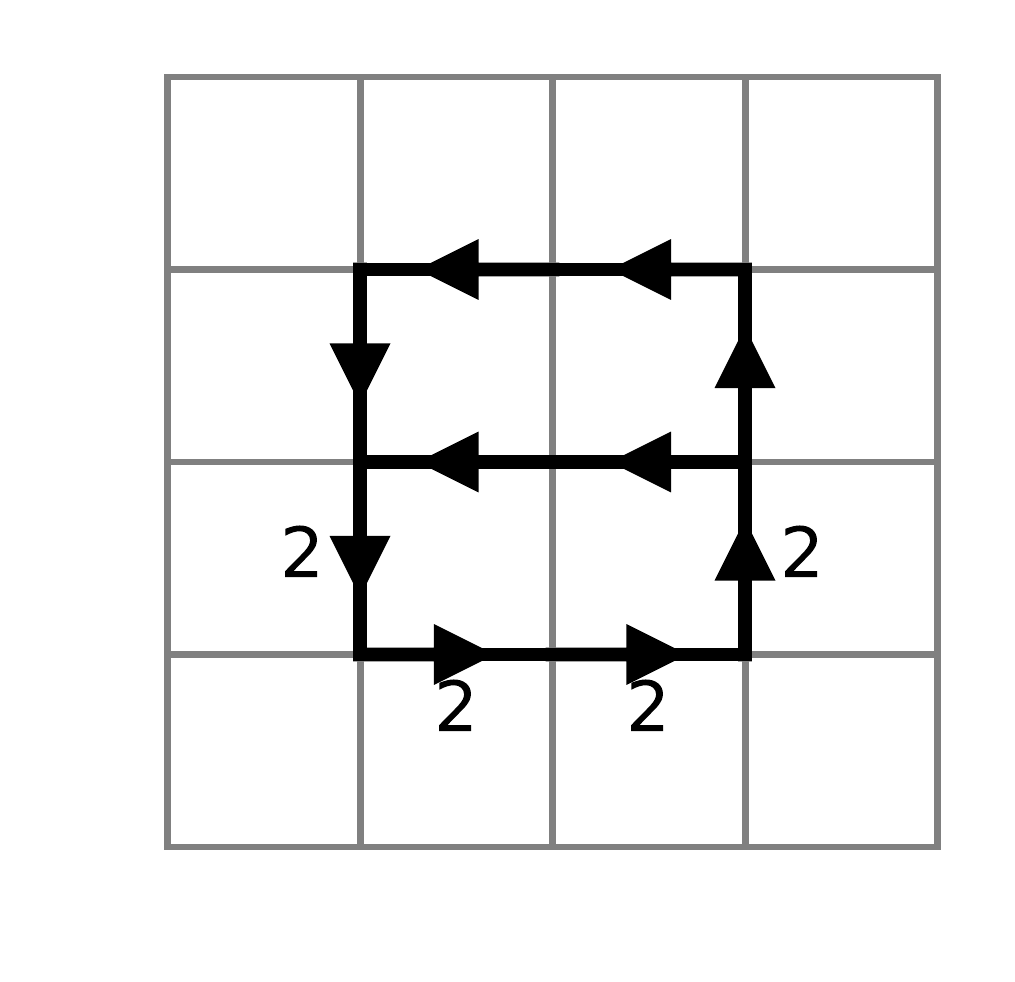}}
  \caption{Flow-firing does not satisfy the diamond property.  In the
    top configuration the two highlighted edges can fire but after one
    of the edges fires the other one can no longer fire.}
  \label{fig:diamondfail}
\end{figure}

\section{The pulse}

In graphical chip-firing, an important example is the \emph{pulse}
on the infinite grid.  The pulse configuration consists
of $n$ chips at the origin and $0$ chips elsewhere in $\Zz^d$.  For
$d=1$, this is chip-firing from a single stack of chips on a line.
The properties of chip-firing from a single stack on the line were
studied in detail in \cite{Spencer} and \cite{BallsWalls}.  The final
configuration is independent of the firing choices made throughout the
process and consists of a single chip in each position of an interval
centered at the origin; the origin itself has zero chips if the
initial number of chips is even.

For $d=2$, chip-firing from a stack of chips at the origin yields a
well known fractal pattern associated with chip-firing; see, e.g.,
\cite{paoletti2014abelian}, \cite{Book}[Chapter 5].  The final
configuration is again independent of the firing choices made
throughout. Yet this final configuration, resulting from the pulse on
the graph of the $2$-dimensional grid, has proved very difficult to
understand. Much work has gone into studying its properties; see,
e.g., \cite{le2002identity}, \cite{levine2009strong}, \cite{Scaling},
\cite{Apollonian2}.

The current paper can be seen as a first step in understanding the
basic properties of fundamental initial configurations (pulses of
flow) on higher dimensional spaces. 

\section{Flow on a single edge / Non-terminating}
\label{sec:single}

Following the graphical chip-firing examples of the pulse, one might
naturally consider an initial flow configuration consisting of a large
amount of flow on a single edge.  However, the flow-firing process
from such an initial configuration does not terminate.

Figure~\ref{fig:singleedge} shows an example of such an initial state
and the configurations resulting from the flow-firing process after
several steps.

\begin{figure}
  \centering
  \begin{tabular}{cccccc}
    & & \includegraphics[clip=true,trim=30 15 0 0, height=1.3in]{figures/F-start.pdf} \\
    \parbox[c]{1.5in}{\includegraphics[clip=true,trim=30 15 0 0, height=1.3in]{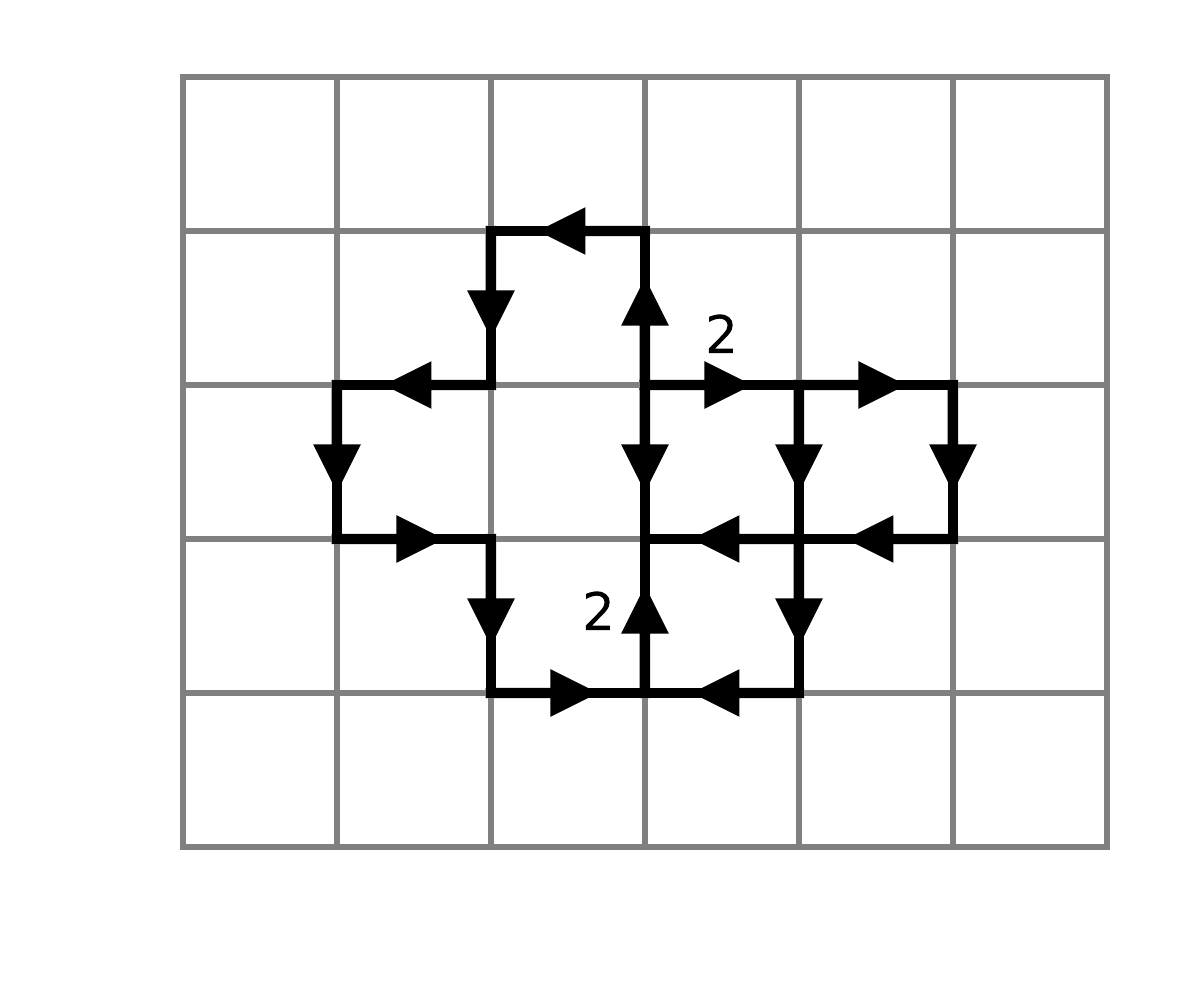}} &
    {\Large $\rightarrow$} &
    \parbox[c]{1.5in}{\includegraphics[clip=true,trim=30 15 0 0, height=1.3in]{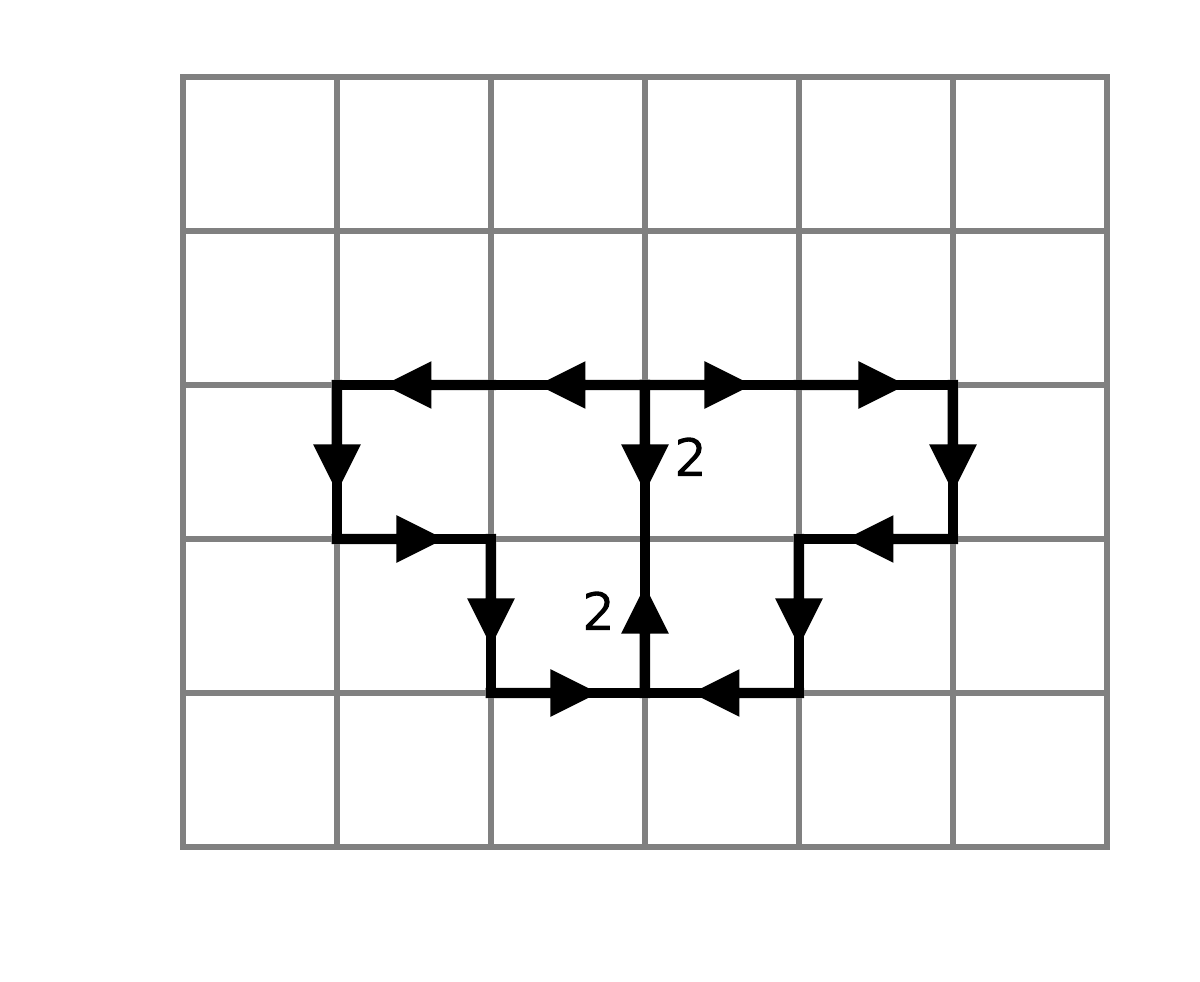}} &
    {\Large $\rightarrow$} &
    \parbox[c]{1.5in}{\includegraphics[clip=true,trim=30 15 0 0, height=1.3in]{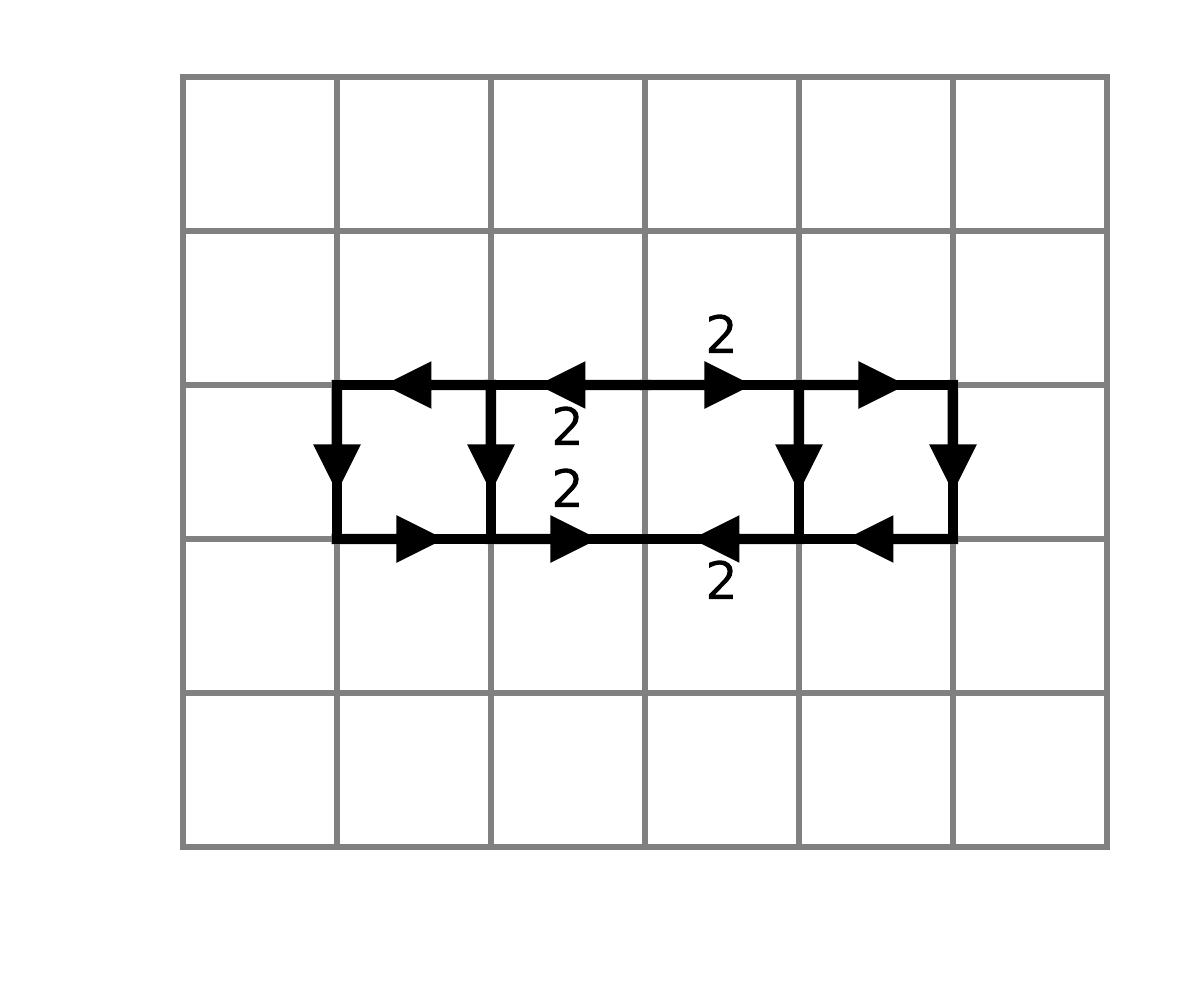}} &
    {\Large $\cdots$}
  \end{tabular}
  \caption{Some intermediate configurations reachable from a single
    edge flow.  This flow-firing process never terminates.}
  \label{fig:singleedge}
\end{figure}

\begin{proposition}
  \label{prop:pigeonhole}
  The flow-firing process on the grid does not terminate from any
  initial configuration which has a vertex $v$ with $|\inflow(v) -
  \outflow(v)| > 4$.
\end{proposition}
\begin{proof}
  Suppose $|\inflow(v)-\outflow(v)| > 4$.  Since $\deg(v) = 4$, by the
  pigeonhole principle, some edge touching $v$ must have at least 2
  units of flow and can fire.  Since $\inflow(v)-\outflow(v)$ is
  conserved by rerouting there will always be an edge touching $v$
  that can fire.
\end{proof}

\section{Conservative flows / Terminating}
\label{sec:terminate}

\begin{definition}
  A flow configuration is \emph{conservative} if for each vertex $v$,
  $\inflow(v) - \outflow(v) = 0$.
\end{definition}

In this section we prove that the flow-firing process initiated at a
conservative flow always terminates in a finite number of steps.
First note that if the initial flow is conservative, it remains
conservative throughout the process.

Importantly, conservative flows allow for a dual representation
consisting of flow on faces.  For this representation, associate
an integer value to each face of the grid instead of each edge.  A
positive value is interpreted as a local clockwise circulation.  A
negative value is interpreted as a local counter-clockwise
circulation.

A flow configuration on the faces induces a flow configuration on the
edges.  For each edge $e$, the flow on $e$ is the sum of the flows
implied by the circulations around the two faces containing $e$.
Furthermore, any \emph{conservative} flow on the edges is induced from
some face configuration.  This follows from the fact that every
conservative flow is a sum of cycles and the boundaries of the faces
of a planar graph span the cycle space of the graph (see, e.g., \cite{FF,KV}).

Figure~\ref{fig:faces} shows an example of a conservative flow and the
corresponding face representation.

\begin{figure}
  \centerline{\parbox[c]{3.2in}{\includegraphics[width=1.5in]{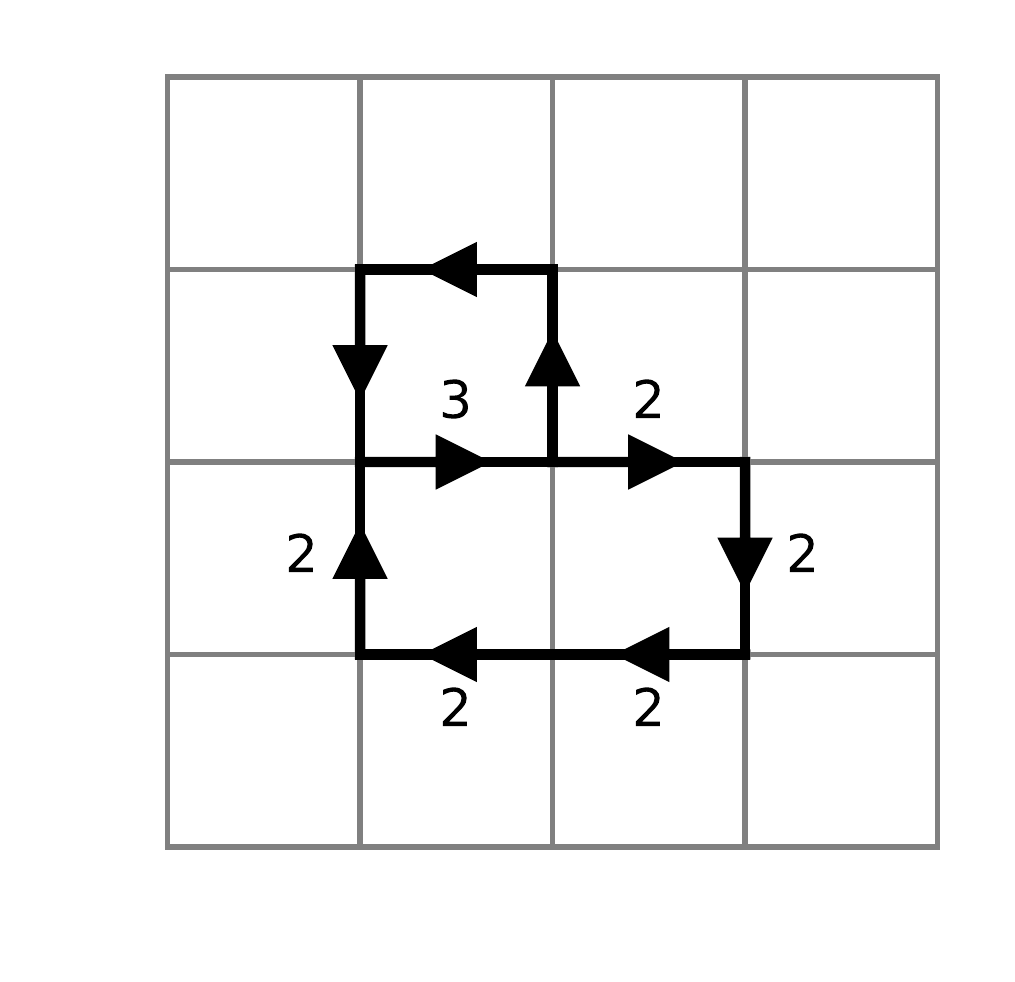}
      \includegraphics[width=1.5in]{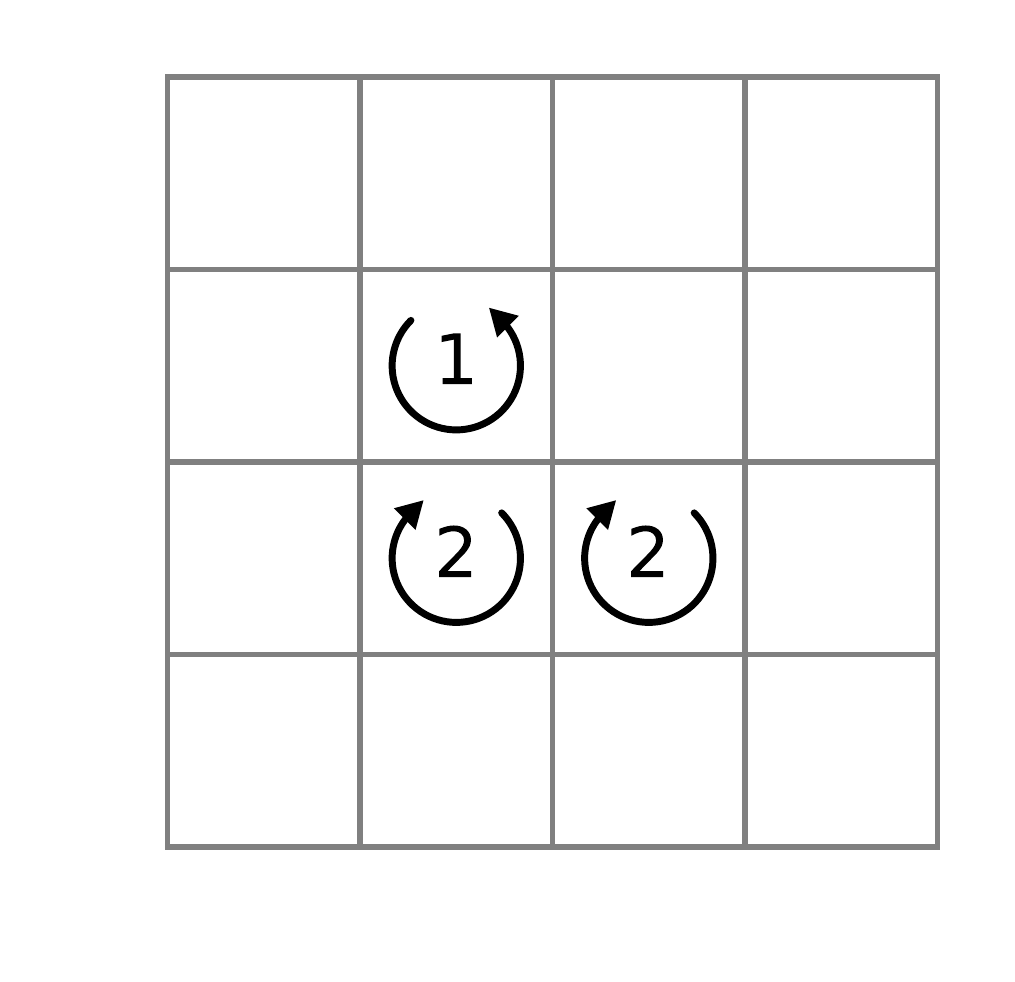}} $F=(\ldots,2,2,\ldots,-1,\ldots)$}
  \caption{A flow configuration and the corresponding face representation.}
  \label{fig:faces}
\end{figure}

Note that the conservative condition is necessary.  A configuration
with flow on a single edge, as considered the previous section, cannot
be represented by a configuration of face circulations.

In the flow-firing process an edge $e$ can fire if it has at least two
units of flow.  In the face representation this means that the values on
the two faces containing $e$ differ by at least two.  
Using a face representation $F$ the flow-firing process can 
be equivalently defined as follows.\\

\fbox{
  \parbox{.9\textwidth}{
    $\,$ \\
    \noindent{\bf The flow-firing process (face representation)}\\
    For the grid graph and conservative initial configuration\\
    $\,$ \\
    At each step:  
    \begin{itemize}  
    \item Choose two neighboring faces $a$ and $b$ with $F_a \ge F_b+2$.
    \item Fire $a$ and $b$ by decreasing $F_a$ by $1$ and increasing $F_b$ by $1$,
       \begin{align*}
        F_a & \rightarrow F_a - 1, \\
        F_b & \rightarrow F_b + 1.
      \end{align*}
    \end{itemize}
    \vspace{.2cm}
}} \\

The definition of the flow-firing process using the face
representation is perhaps more natural in comparison to graphical
chip-firing.  One can picture stacks of ``circulation chips'' on the
faces of $G$.  A firing move sends a chip from one stack to a
neighboring smaller stack.  A significant difference from graphical
chip-firing is that in flow-firing with the face representation,
circulation chips move to a single neighbor not to all neighbors at
the same time.

\begin{theorem}
  \label{thm:conservative}
  The flow-firing process on the grid starting from a finite
  conservative flow terminates after a finite number of steps.
\end{theorem}
\begin{proof}
Let $f$ be a finite conservative flow on the edges of $G$.  Let $F$ be
the corresponding face representation.  Define the potential function
$$\phi(F) = \sum F_\sigma^2,$$ which is an infinite sum over all faces
of $G$ with finite non-zero support.

Suppose that neighboring faces $a$ and $b$ fire and that $F_a > F_b$.
Call the resulting configuration $F'$.  We have,
\begin{align*}
  F'_a &= F_a - 1, \\
  F'_b &= F_b + 1,
\end{align*}
and $F'_c=F_c$ at all other faces $c$.  The difference in potential is:
\begin{align*}
  \phi(F)-\phi(F') 
  &= [(F_a)^2+(F_b)^2] -[(F_a-1)^2 + (F_b+1)^2] \\
  &= (F_a)^2+(F_b)^2 - (F_a)^2 - (F_b)^2 + 2F_a  -2F_b -2\\
  &= 2F_a - 2F_b - 2\\
  &= 2(F_a - F_b - 1)\\
  &\ge 2.
\end{align*}

The last inequality follows from the fact that $F_a - F_b \geq 2$ or
else faces $a$ and $b$ could not fire.

The potential function $\phi$ is non-negative and strictly decreases
with each flow-firing step.  Therefore, starting from any
configuration with finite potential, the process must terminate in a
finite number of steps.
\end{proof}
  
Again, note that this argument does not apply to non-conservative
flows, such as the configurations with flow on a single edge
considered in the last section.  Non-conservative flows do not afford
a face representation and therefore the potential function $\phi$
cannot be defined.

Within the class of conservative flows a possible analog to the pulse
is a configuration with a large circulation around a single face.  In
terms of the face representation this corresponds to a large stack of
positive circulation chips on a single face.

\begin{corollary}
  The flow-firing process starting from $k$ units of flow around a
  single face terminates after a finite number of steps.
\end{corollary}

Figure~\ref{fig:tetris} illustrates the result of the flow-firing
process starting from $k=4$ units of flow around a face.  While
the process always terminates, there are many possible final
configurations.

\begin{figure}
  \centering

  \includegraphics[clip=true,trim=30 15 0 0, height=1.2in]{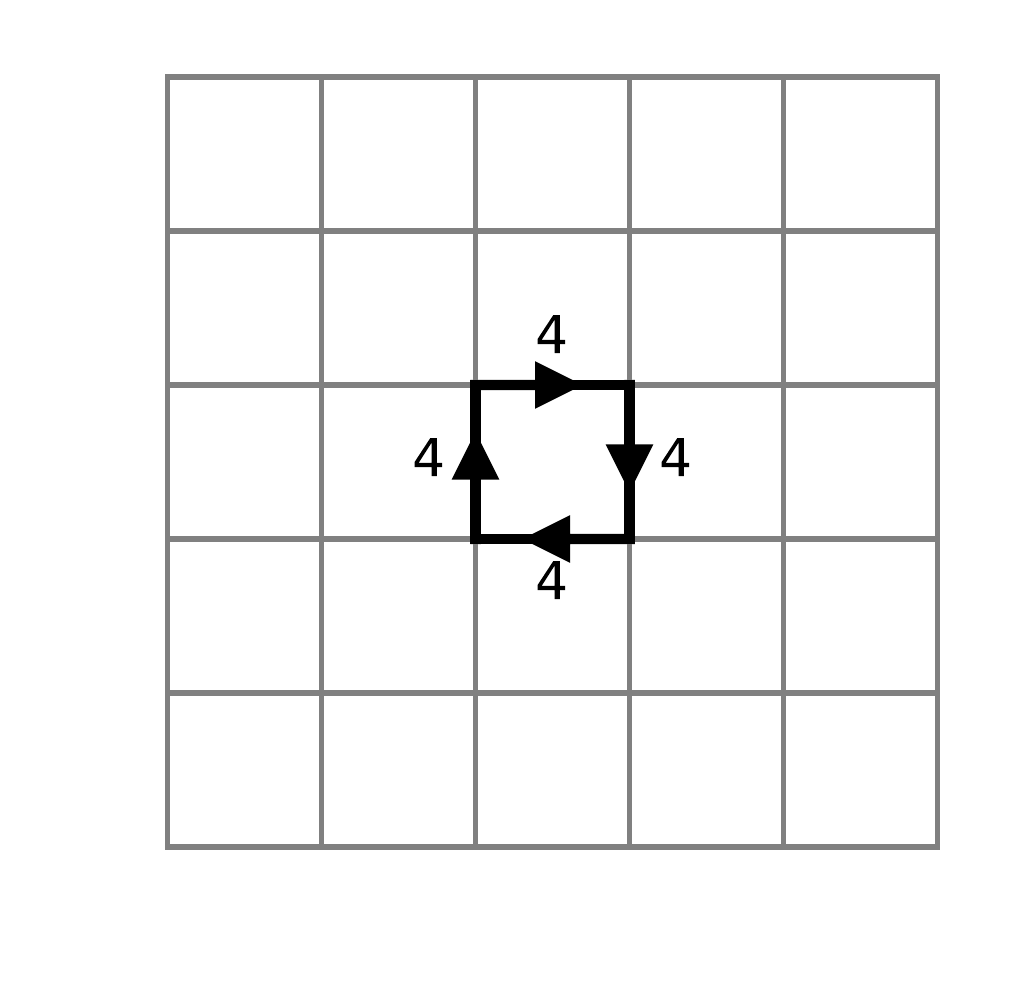}

  {\Large $\swarrow$ \hspace{.75cm} $\downarrow$ \hspace{.75cm} $\searrow$}
  
  {\huge $\cdots$}
  \parbox[c]{1.2in}{\includegraphics[clip=true,trim=30 15 0 0, height=1.2in]{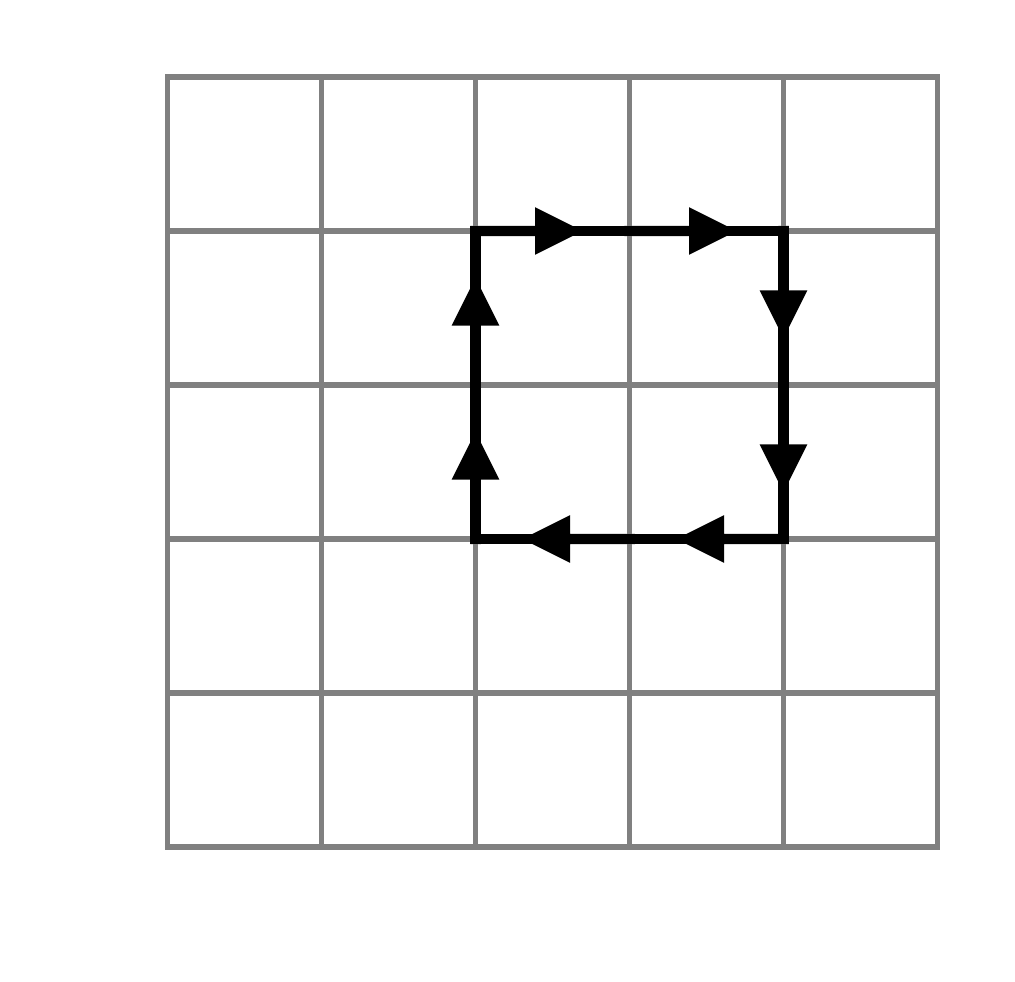}}
  \parbox[c]{1.2in}{\includegraphics[clip=true,trim=30 15 0 0, height=1.2in]{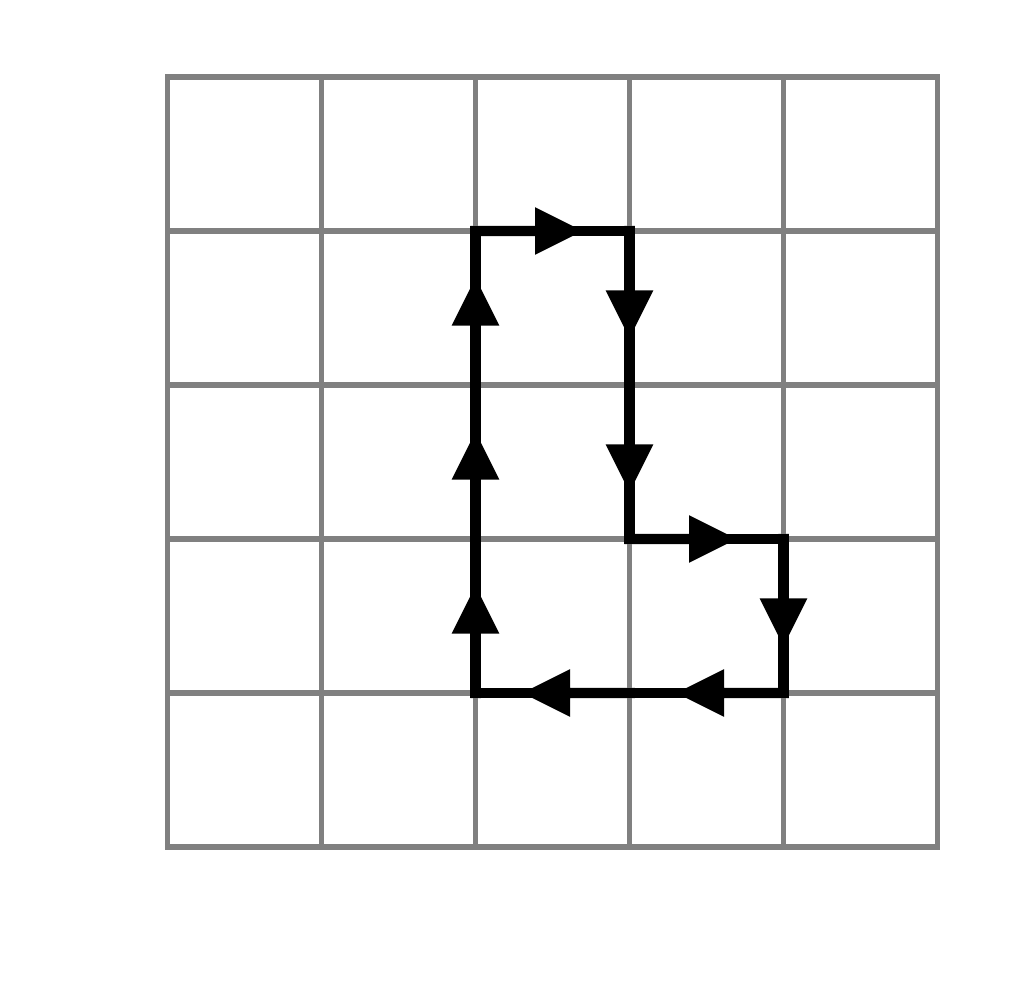}} 
  \parbox[c]{1.2in}{\includegraphics[clip=true,trim=30 15 0 0, height=1.2in]{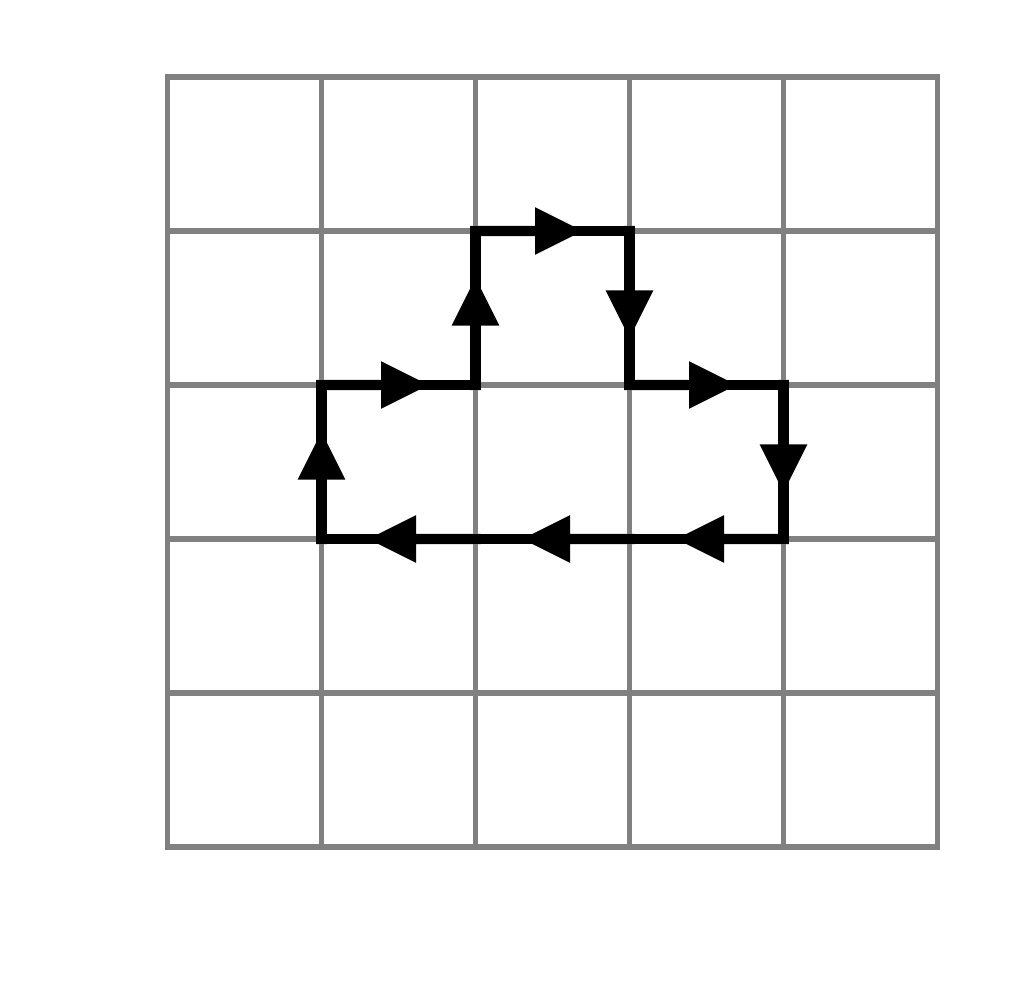}}   
  {\huge $\cdots$}

  (a) Edge representation.
  
  \vspace{.5cm}
  \includegraphics[clip=true,trim=30 15 0 0, height=1.2in]{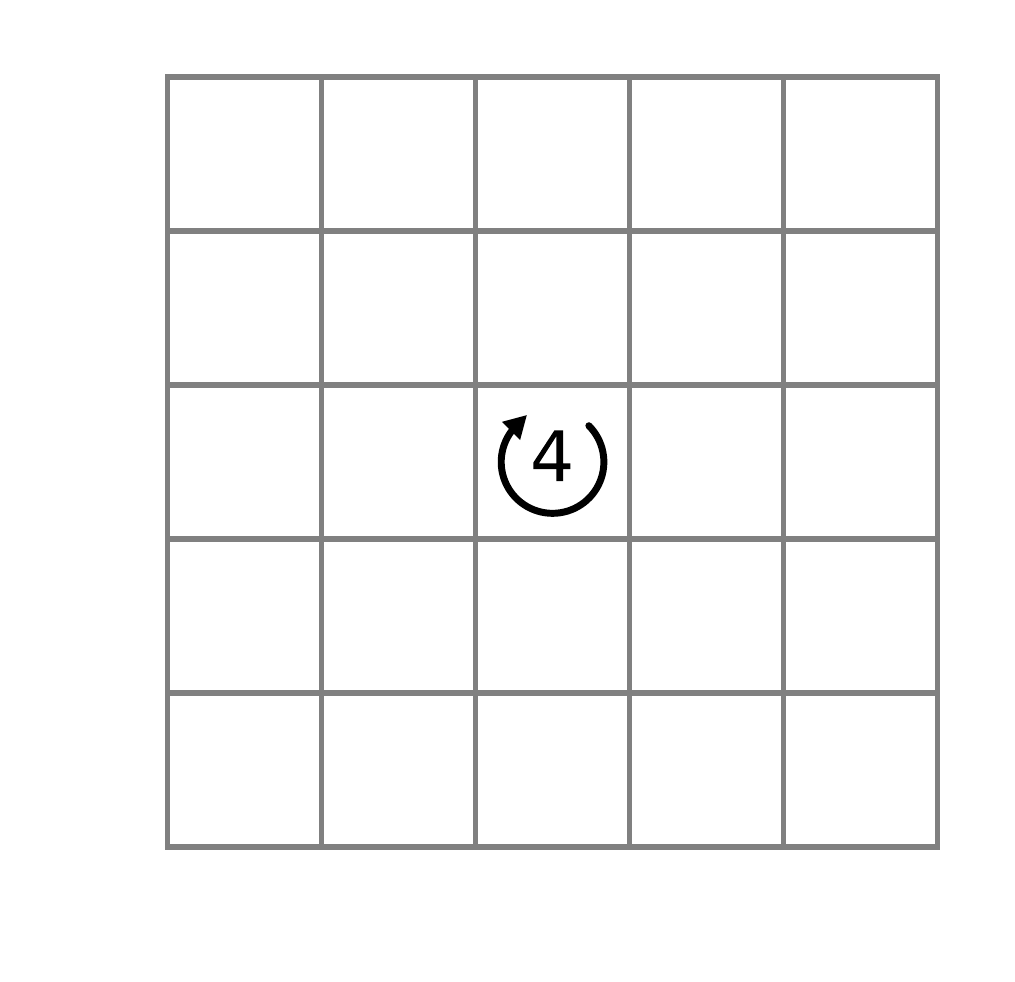}

  {\Large $\swarrow$ \hspace{.75cm} $\downarrow$ \hspace{.75cm} $\searrow$}
  
  {\huge $\cdots$}
  \parbox[c]{1.2in}{\includegraphics[clip=true,trim=30 15 0 0, height=1.2in]{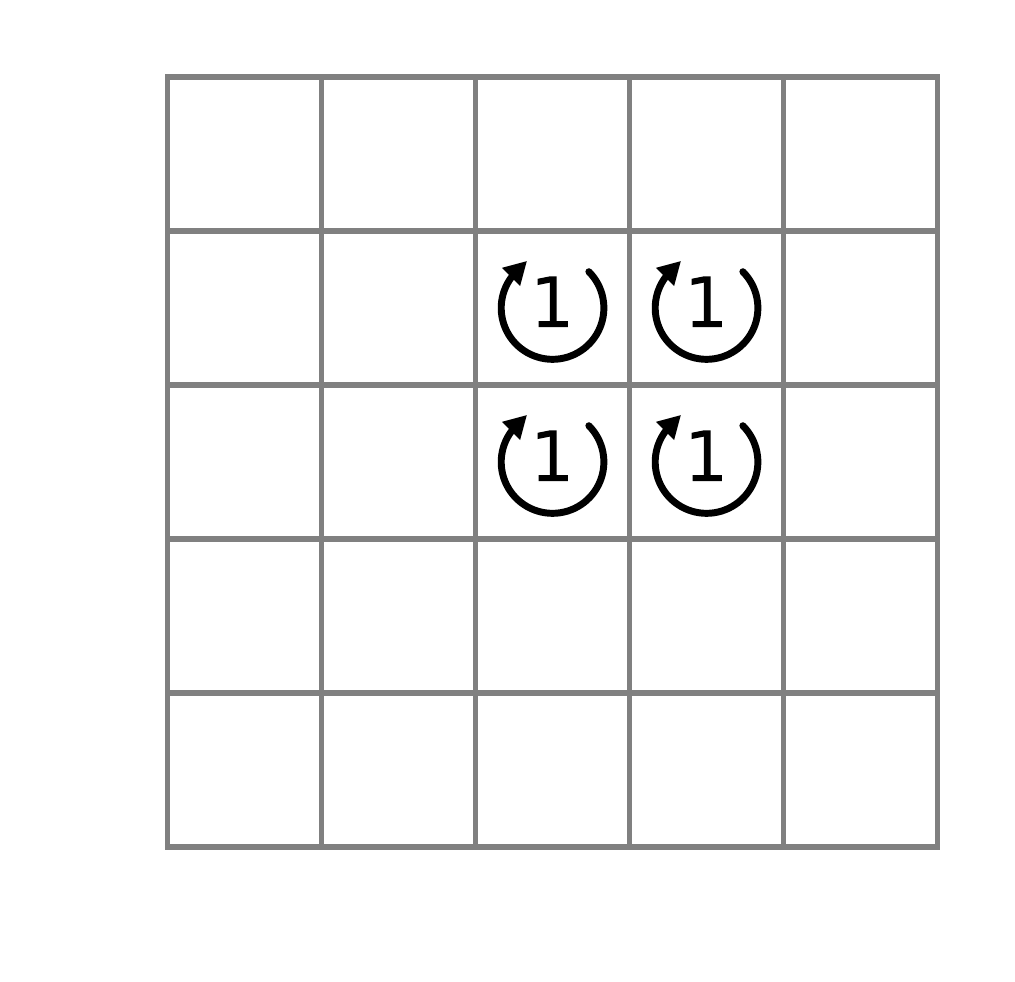}}
  \parbox[c]{1.2in}{\includegraphics[clip=true,trim=30 15 0 0, height=1.2in]{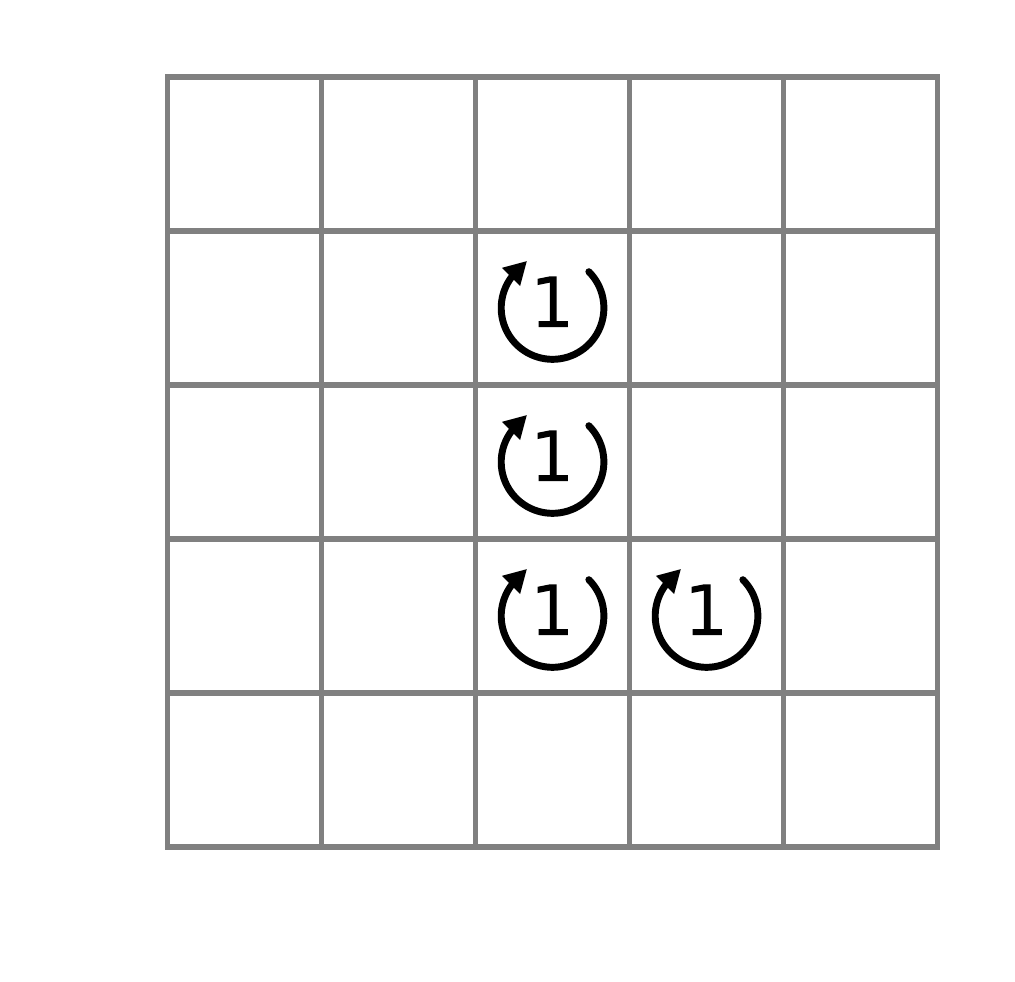}} 
  \parbox[c]{1.2in}{\includegraphics[clip=true,trim=30 15 0 0, height=1.2in]{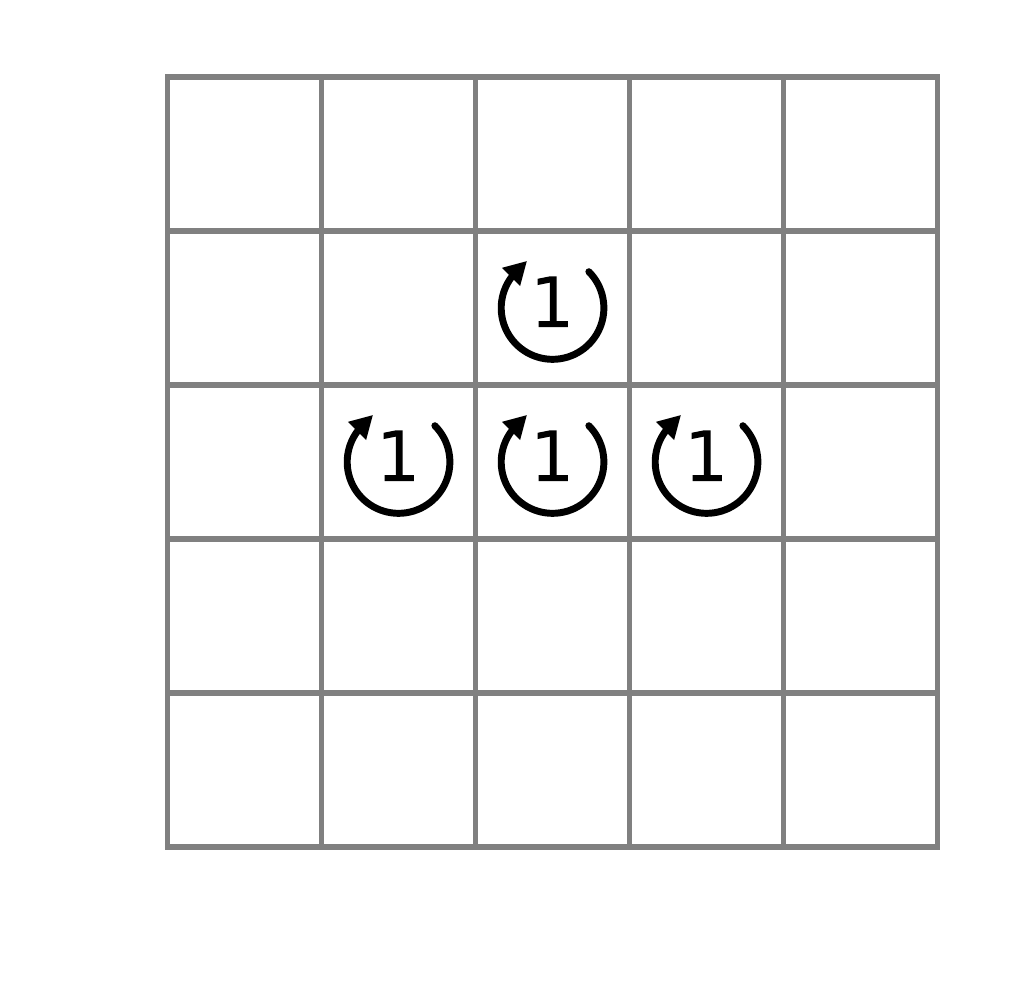}}   
  {\huge $\cdots$}

  (b) Face representation.
  
\caption{Starting with $k=4$ units of flow around a face always
  terminates but there are many possible final configurations.}
\label{fig:tetris}
\end{figure}

\section{Conservative flows around a boundary / Confluent}
\label{sec:confluent}

We next consider the grid graph with a distinguished face (square).
In terms of flow rerouting, the distinguished face behaves like an
obstruction or hole -- flow on an edge incident to the distinguished
face cannot divert across the distinguished face.  In terms of the
face representation (for conservative flows), the distinguished face
behaves like a source or sink -- flow on adjacent faces determines
which behavior is seen.  A value at the distinguished face can be
thought of as a boundary condition for the flow-firing process.

Formally, let $G$ be the grid graph again embedded as $\mathbb{Z}^2$ and
let $\sigma$ be a fixed face of $G$. Define the following process.\\

\fbox{
  \parbox{.9\textwidth}{
    $\,$ \\
    \noindent{\bf The flow-firing process (edge representation)}\\
    For the grid graph with distinguished face $\sigma$\\
    $\,$\\
    At each step:  
    \begin{itemize}
    \item Choose an edge $e$.
    \item If $e \not\subset \sigma$ and $e$ has $2$ or more units of
      flow (in either direction) then 1 unit of \\ flow is rerouted around each of the two faces containing $e$.
    \item If $e \subset \sigma$ and  $e$ has $1$ or more units of
      flow (in either direction) then 1 unit of \\ flow is rerouted around the unique
      face not equal to $\sigma$ containing $e$.
    \end{itemize}
    \vspace{.2cm}
}} \\

For conservative flows we have the following equivalent description of
the process using the face representation introduced in Section~\ref{sec:terminate}.\\

\fbox{
  \parbox{.9\textwidth}{
    $\,$ \\
    \noindent{\bf The flow-firing process (face representation)}\\
    For the grid graph with distinguished face $\sigma$ and conservative initial configuration\\
    $\,$ \\
    At each step:
    \begin{itemize}
    \item Choose two neighboring faces $a$ and $b$.
    \item If $a \neq \sigma$, $b \neq \sigma$ and
      $F_a \ge F_b-2$, decrease $F_a$ by $1$ and increase $F_b$ by $1$,
      \begin{align*}
        F_a & \rightarrow F_a - 1, \\
        F_b & \rightarrow F_b + 1.
      \end{align*}
    \item If $b = \sigma$ and $F_{\sigma} > F_a$, increase $F_a$ by $1$,
      \begin{align*}
        F_a & \rightarrow F_a + 1. 
      \end{align*}
    \item If $b = \sigma$ and $F_{\sigma} < F_a$, decrease $F_a$ by $1$,
      \begin{align*}
        F_a & \rightarrow F_a - 1. 
      \end{align*}
    \end{itemize}
    \vspace{.2cm}
}} \\

In the first case (when $a$ and $b$ are both not equal to $\sigma$) we
say a circulation chip moves from $a$ to $b$.  In the second case we
say a circulation chip is created at $a$.  In the third case we say a
circulation chip is deleted from $a$.

Write $(G,\sigma)$ for the grid graph with distinguished face
$\sigma$.

\begin{proposition}
  \label{prop:maxmin}
  Under the flow-firing process with the face representation for $(G,\sigma)$:
\begin{enumerate}
\item\label{a} The maximum value over all faces does not increase. 
\item\label{b} The minimum value over all faces does not decrease. 
\item\label{c} The value at $\sigma$ does not change. 
\end{enumerate}
\end{proposition}
\begin{proof}
For (\ref{a}), note that all moves that increase the value of a face
involve a face of greater value, therefore the maximum value in a
configuration cannot increase.  Part (\ref{b}) is analogous.
Statement (\ref{c}) is the observation that the flow-firing rules
never alter the value of $F_{\sigma}$.
\end{proof}

From Proposition~\ref{prop:maxmin} part (\ref{b}) we see that starting
from a configuration of positively oriented face circulations we can
only ever generate configurations of positively oriented face
circulations.

For the remainder of the section, we consider the specific initial
configuration consisting of $k$ units of flow around $\sigma$.  The
face representation, $K$, for this configuration is,
$$K_\sigma = k \textrm{ and } K_\tau = 0 \textrm{ for all } \tau \neq \sigma.$$

Figure~\ref{fig:pyramid} shows the result of flow-firing starting from
$K$.  Surprisingly, as we show next, there is a unique final
configuration in this case.

\begin{figure}
  \centering
  \begin{tabular}{c}
  \parbox[c]{3in}{\includegraphics[clip=true,trim=30 30 0 0, height=3in]{figures/H-start.pdf}}
  {\Large $\rightarrow$}
  \parbox[c]{3in}{\includegraphics[clip=true,trim=30 30 0 0, height=3in]{figures/H-end.pdf}} \\
  \parbox[c]{3in}{\includegraphics[clip=true,trim=30 30 0 0, height=3in]{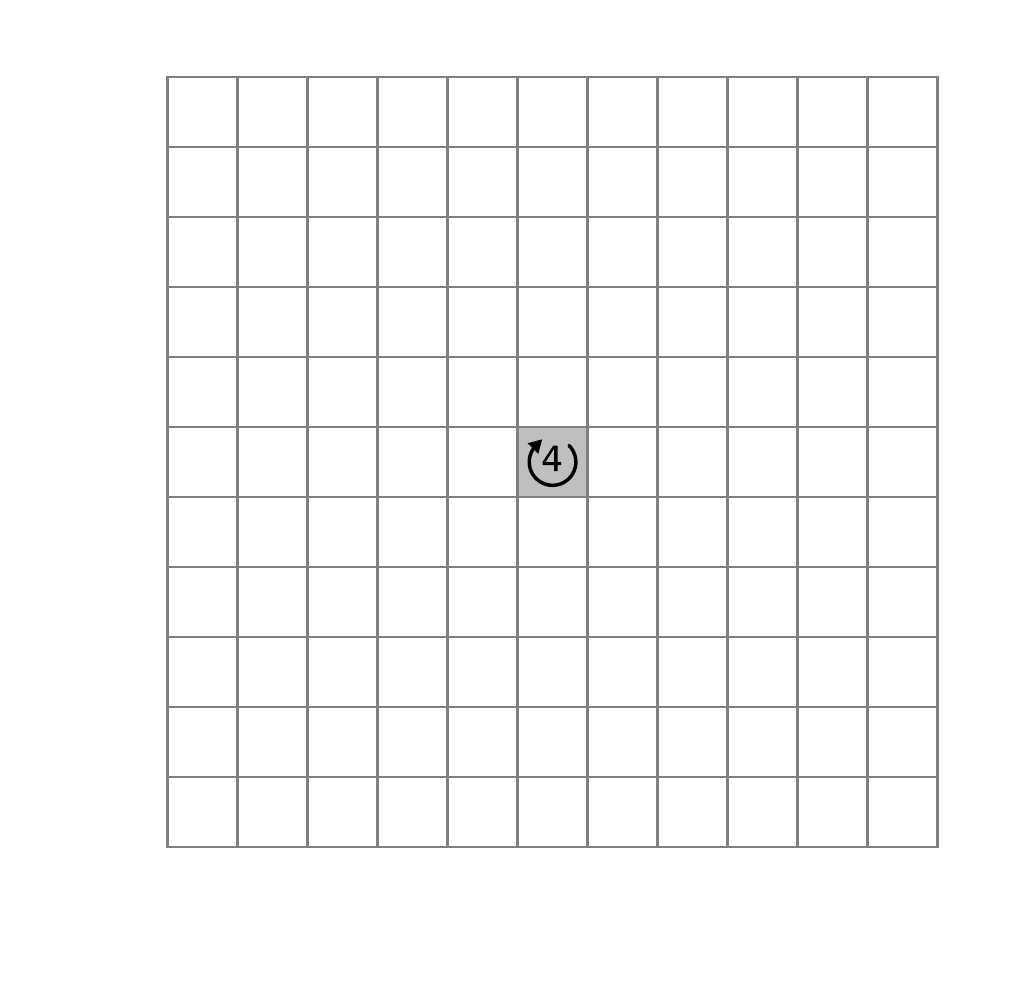}}
  {\Large $\rightarrow$}
  \parbox[c]{3in}{\includegraphics[clip=true,trim=30 30 0 0, height=3in]{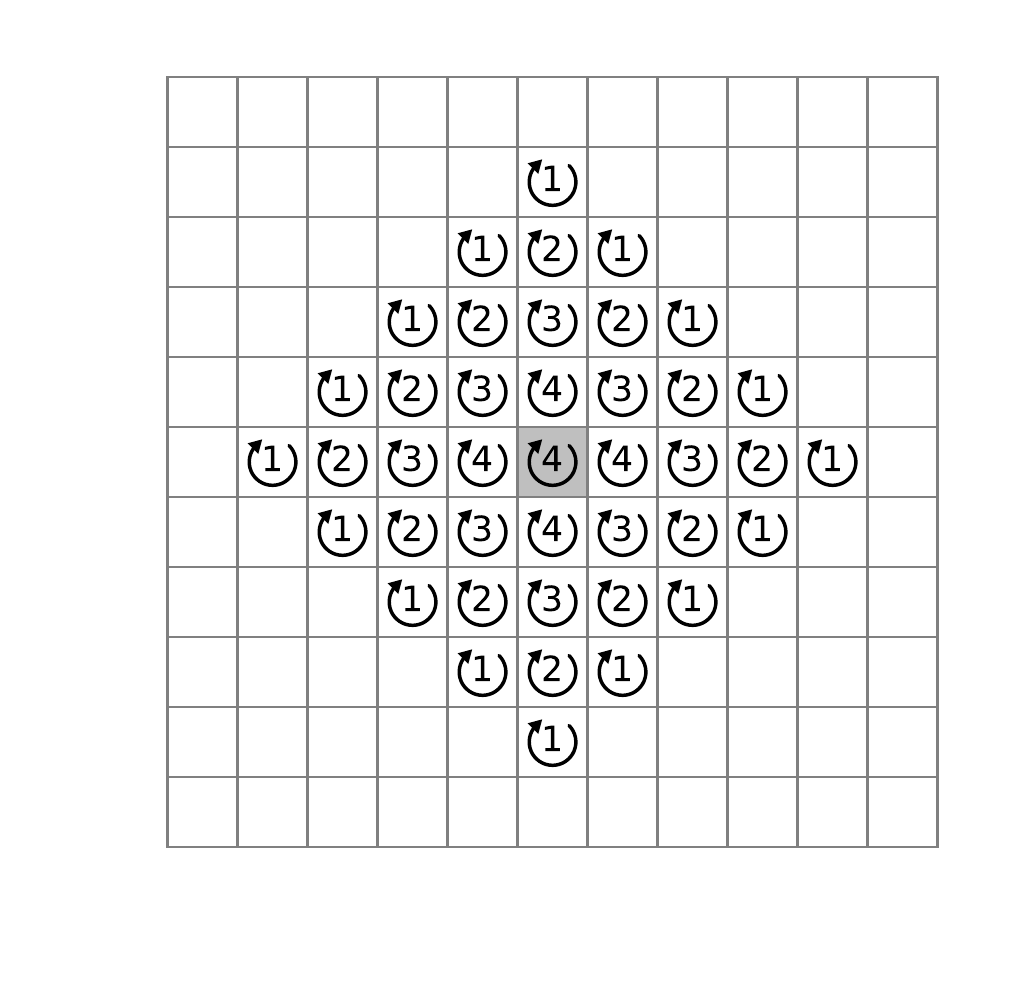}}  
\end{tabular}
\caption{Flow-firing starting with a configuration of $k=4$ units of
  flow around a distinguished face (a hole).  The top shows the edge representation and the
  bottom shows the face representation of the initial and final
  configurations.}
\label{fig:pyramid}
\end{figure}

Define $\dist(\sigma,\tau)$ to be the distance from $\sigma$ to $\tau$
in the dual graph of $G$.  For the grid graph this is the Manhattan
distance.

\begin{lemma}
  Let $K^*$ denote any configuration reachable from $K$ via the
  flow-firing process for $(G,\sigma)$.  Then for all faces
  $\tau\neq\sigma$ of $G$,
  $$K^*_{\tau} \le \max\{0, k - \dist(\sigma,\tau)+1\}.$$
\label{lem:bound}
\end{lemma}
\begin{proof}
  We proceed by induction on $\dist(\sigma,\tau)$.

  Base case: When $\dist(\sigma,\tau)=1$ the result follows from the fact
  that the maximum value in $K$ is $k$ and the maximum value cannot
  increase.

  Induction step: Suppose the claim holds for all faces with distance
  at most $d-1$ from $\sigma$.  Let $A = \{ a \;|\; \dist(\sigma,a)
  \ge d \}$ be the set of faces with distance at least $d$ from
  $\sigma$.  Initially, $K_a = 0$ for all $a \in A$.  Suppose $K^*_a
  \not\le \max\{0,k-d+1\}$ for some $a \in A$.  Consider, in
  particular, the first time that $K^*_a > \max\{0,k-d+1\}$ for some
  $a \in A$.  The face $a$ must have just received a circulation chip
  from a neighboring face $b$ with $K^*_b > \max\{0,k-d+1\}+1$ before
  the last step.  Since this is the first time $K^*_a >
  \max\{0,k-d+1\}$ for $a \in A$, the face $b$ cannot be in $A$.
  Since $b$ is a neighbor of a face in $A$ and not in $A$, it must be
  that $\dist(\sigma,b)=d-1$.  But, by induction, the value at $b$
  must be at most $\max\{0,k-d+2\} \le \max\{0,k-d+1\}+1$ which is a
  contradiction.
\end{proof}

The main result of this section, Theorem~\ref{thm:main}, shows that
starting from the initial configuration $K$, the flow-firing process
always terminates at the configuration achieving equality for all
bounds in Lemma~\ref{lem:bound}.  First we need the following
observations.

\begin{proposition}
  \label{prop:1415}
Let $K^*$ denote any configuration reachable from $K$.  Then
  \begin{enumerate}
\item\label{14} $K^*_{\tau} \ge 0$ for all $\tau$.
\item\label{15} The total number of circulations chips, $\sum
  K^*_{\tau}$, is bounded and non-decreasing over time.
  \end{enumerate}
\end{proposition}
  \begin{proof}
    (\ref{14}) This follows from Proposition~\ref{prop:maxmin}
    part~(\ref{b}) since all values are non-negative in the initial
    configuration.

    (\ref{15}) For any reachable configuration, $K^*_{\sigma} = k$
    and $K^*_{\tau} \le \max\{0,k - \dist(\sigma,\tau)+1\}$ for $\tau \neq
    \sigma$, thus the total sum is bounded.  The sum is non-decreasing
    because no circulation chips are ever deleted.  Neighbors of
    $\sigma$ always have value at most $k$ by Lemma~\ref{lem:bound},
    and $\sigma$ always has value $k$.  Therefore neighbors of
    $\sigma$ never have value larger than $\sigma$.
\end{proof}

\begin{theorem}
  \label{thm:main}
  The flow-firing process on $(G,\sigma)$ with initial configuration
  $K$ terminates at a unique configuration $K^{\bullet}$ after a finite
  number of steps.  The final configuration has face representation
  $$K^{\bullet}_\sigma = k \textrm{ and } K^{\bullet}_{\tau} = \max\{0, k -
  \dist(\sigma,\tau)+1\} \textrm{ for all } \tau \neq \sigma.$$ 
\end{theorem}
  
\begin{proof}
  First, we prove that the process stops.  Let $K^*$ be a
  configuration reachable from $K$.  Define the potential function
  $$\psi(K^*) = \sum_{\tau} (k-K^*_{\tau})^2,$$ where the sum is
  over all faces with distance at most $k+1$ from $\sigma$.  Note that
  this function is bounded from below, i.e. $\psi(K^*) \ge 0$.
  Moreover, $\psi(K^*)$ is finite for the initial configuration
  $K^*=K$.  Each flow-firing step decreases $\psi(K^*)$ by at least one:

  For a step that creates a circulation chip at a face $\tau$ neighboring
  $\sigma$,  $K^*_{\tau}$ is always at most $k$.  Therefore adding a
  circulation at $\tau$ can only decrease $(k-K^*_{\tau})^2$.

  For a step that moves a circulation chip from $\tau$ to $\gamma$: Let $F$
  be the configuration before the step and $G$ be the
  configuration after the step.  Then

  \begin{align*}
    \psi(F)-\psi(G) 
    &= [(k-F_{\tau})^2+(k-F_{\gamma})^2]-[(k-F_{\tau}-1)^2+(k-F_{\gamma}+1)^2] \\
    &= [k^2+ F_{\tau}^2 - 2kF_{\tau}+k^2+F_{\gamma}^2-2kF_{\gamma}] \\
    & \, \hspace{1in} -[k^2+(F_{\tau}-1)^2-2k(F_{\tau}-1)+k^2+(F_{\gamma}+1)^2-2k(F_{\gamma}+1)] \\
    &= [F_{\tau}^2 - 2kF_{\tau} + F_{\gamma}^2 - 2kF_{\gamma}] \\
    & \, \hspace{1in} -[F_{\tau}^2 +1-2F_{\tau}-2kF_{\tau}+2k+ F_{\gamma}^2 + 1 + 2F_{\gamma}-2kF_{\gamma}-2k] \\
    &= 2({F_{\tau}}-F_{\gamma})-2\\
    &\ge 2, 
  \end{align*}
  where the final inequality follows from the fact that $F_\tau-F_\gamma
  \ge 2$ for a circulation chip to move from $\tau$ to $\gamma$.  \\

  Next, $K^\bullet_\sigma = k$ since the value at $\sigma$ never
  changes.  To see that $K^{\bullet}_{\tau} = \max\{0, k -
  \dist(\sigma,\tau)+1\}$ for $\tau \neq \sigma$, we argue by induction on
  $\dist(\sigma,\tau)$.

  Base case: When $\dist(\sigma,\tau)=1$ we have that $\tau$ is a
  neighbor of $\sigma$.  Because of the allowable firing steps the
  process can only terminate if $K^\bullet_\tau = K^\bullet_\sigma =
  k$.
  
  Induction step: Suppose $\dist(\sigma,\tau)=d>1$.  Let $\gamma$ be a
  neighbor of $\tau$ with $\dist(\sigma,\gamma)=d-1$.  By induction
  $K^\bullet_\gamma = \max\{0,k-d+2\}$.  Because of the allowable
  firing steps the process can only terminate if $K^\bullet_\tau$ 
  is in $\{K^\bullet_\gamma-1,K^\bullet_\gamma,K^\bullet_\gamma+1\}$.
  By Lemma~\ref{lem:bound} it must be that $K^\bullet_\tau \leq
  \max\{0,k-d+1\}$.  Considering the two possible values for
  $K^\bullet_{\gamma}$ and the three possible values for $K^\bullet_{\tau}$ directly
  gives that $K^\bullet_\tau$ must equal $\max\{0,k-d+1\}$.
  
\end{proof}

Figure~\ref{fig:pyramid} shows a pulse with $k=4$ units of flow and the
resulting final configuration.  In terms of the face representation,
the final configurations is an ``Aztec pyramid''.  The number of
circulation chips at $\sigma$ and neighbors of $\sigma$ is $k$.  The number 
of the circulation chips decreases linearly with the $\ell_1$ distance
from $\sigma$ until reaching zero.  In terms of the edge
representation, the final configuration has exactly one unit of flow
on every edge not in $\sigma$ that is in a face within a $\ell_1$-ball
of radius $k$ centered at $\sigma$.  The remaining edges have no flow.

\section{Extensions}
\label{sec:extensions}

As mentioned in the introduction, we study the grid for simplicity but
the results described here can be extended to more general settings.

\subsection{Planar graphs}

The results for the grid carry over essentially unchanged to any
infinite planar graph.

Proposition~\ref{prop:pigeonhole} follows as:  If there is a vertex
$v$ with $$|\inflow(v) - \outflow(v)| > \deg(v)$$ then the flow-firing
process does not terminate.

Theorem~\ref{thm:conservative} follows unchanged: If the initial
configuration is a finite conservative flow then the flow-firing
process terminates after a finite number of steps.

Theorem~\ref{thm:main} also follows unchanged: If the initial
configuration is a circulation around a topological obstruction
$\sigma$ then the flow-firing process terminates in a finite number of
steps at a \emph{unique} final configuration, see
Figure~\ref{fig:planar}.

The final configuration in Theorem~\ref{thm:main} stated in terms of
the face representation is the same for planar graphs.  But in the edge
representation of the final configuration for a general planar graph
not every edge within some radius
of the distinguished face will terminate with exactly one unit of
flow.  If two neighboring faces have the same distance to $\sigma$
then the edge between them will have zero flow in the final
configuration (see Figure~\ref{fig:planar}).

\begin{figure}
  \centering
  \parbox[c]{2.6in}{\fbox{\includegraphics[clip=true,trim=65 60 50 50, width=2.5in]{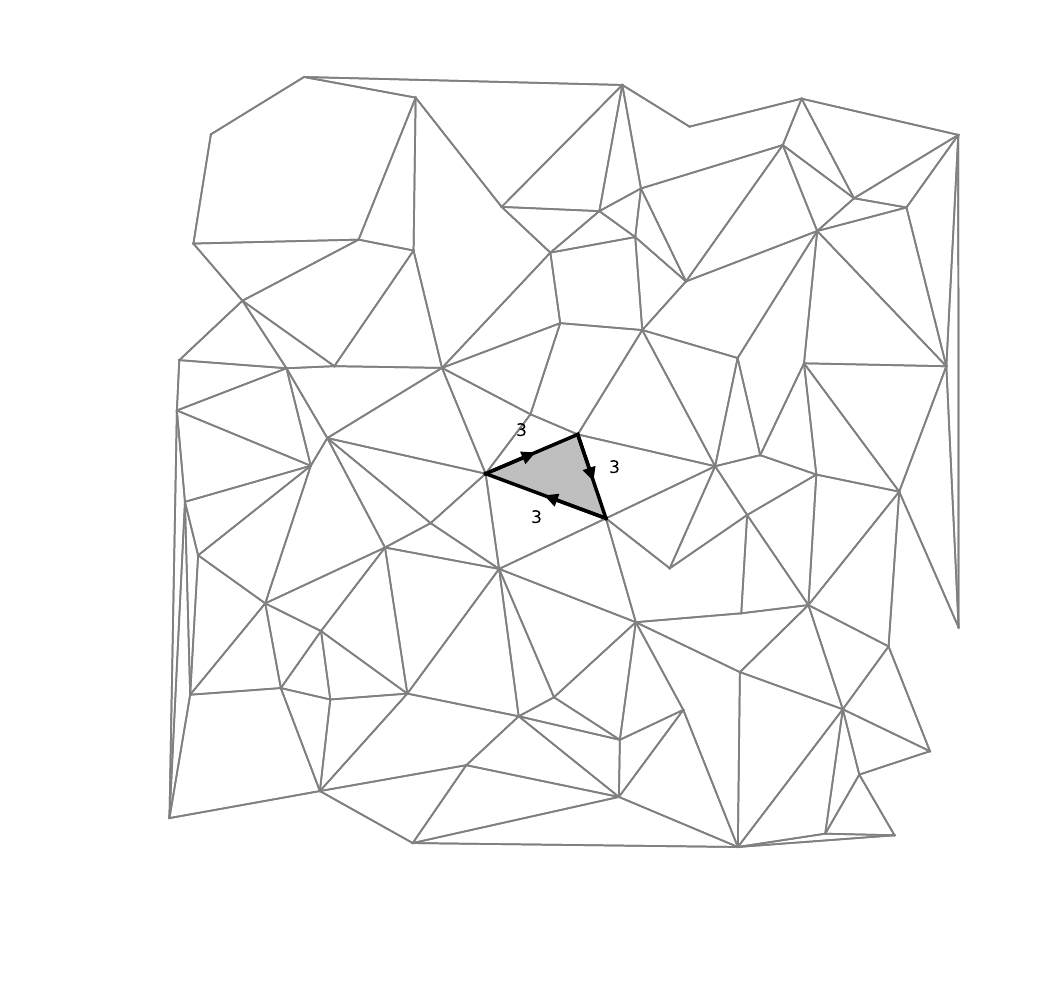}}}
  {\Large $\;\rightarrow\;$}
  \parbox[c]{2.6in}{\fbox{\includegraphics[clip=true,trim=65 60 50 50, width=2.5in]{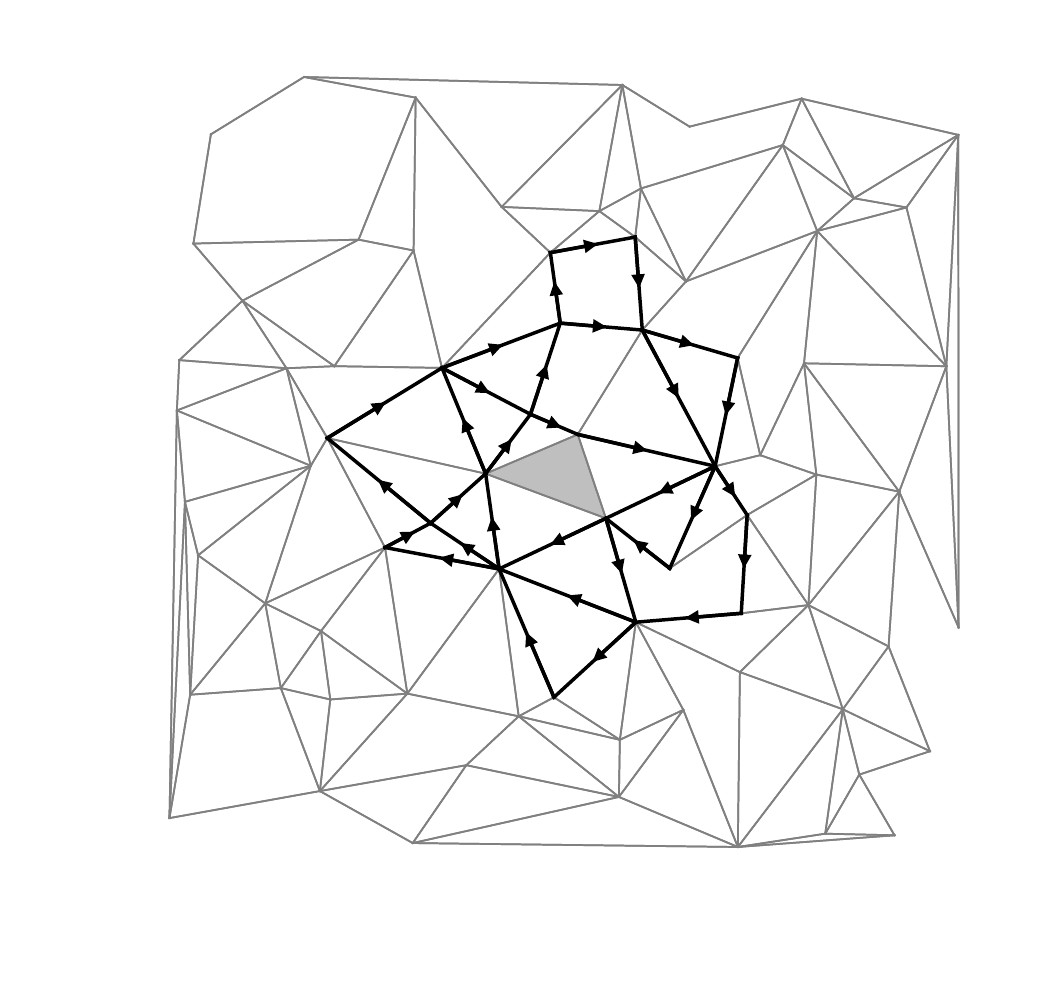}}}
  \vspace{0.5cm}
  
  \parbox[c]{2.6in}{\fbox{\includegraphics[clip=true,trim=65 60 50 50, width=2.5in]{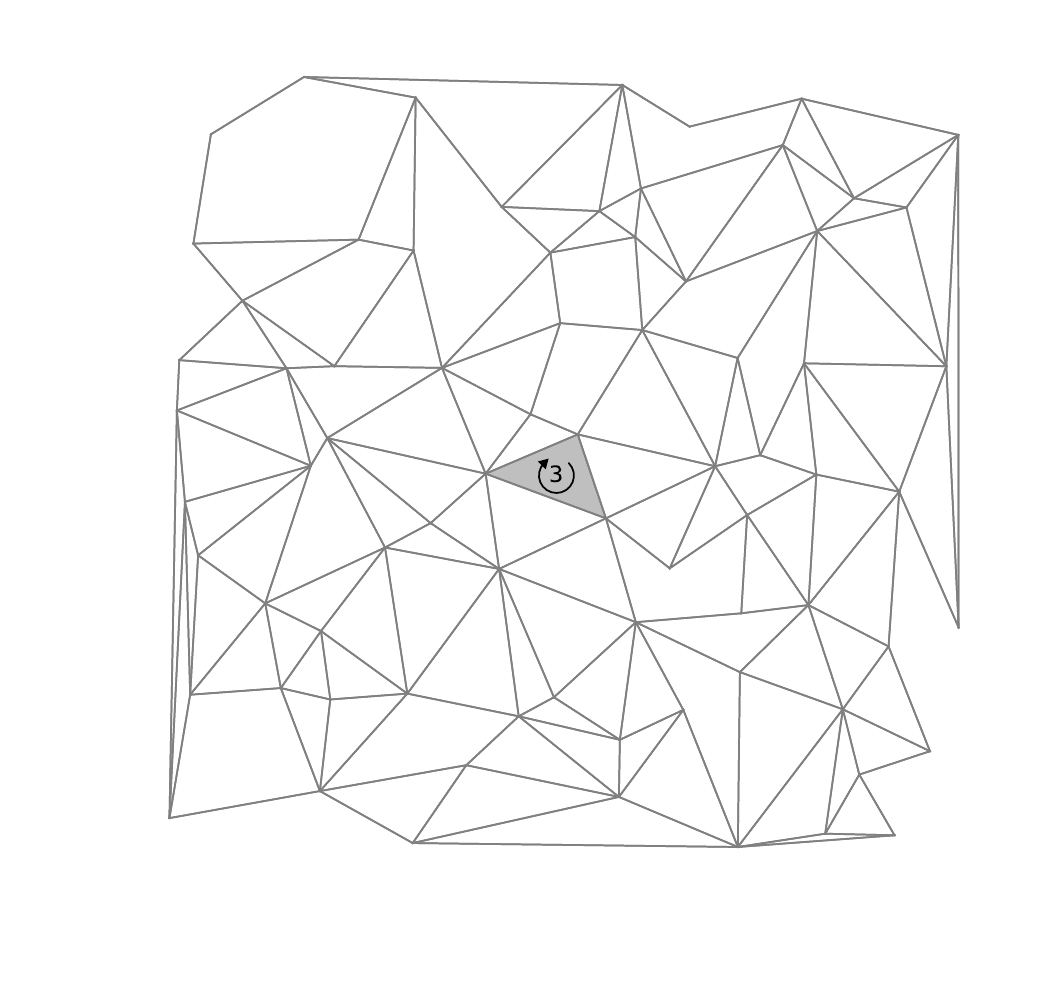}}}
  {\Large $\;\rightarrow\;$}
  \parbox[c]{2.6in}{\fbox{\includegraphics[clip=true,trim=65 60 50 50, width=2.5in]{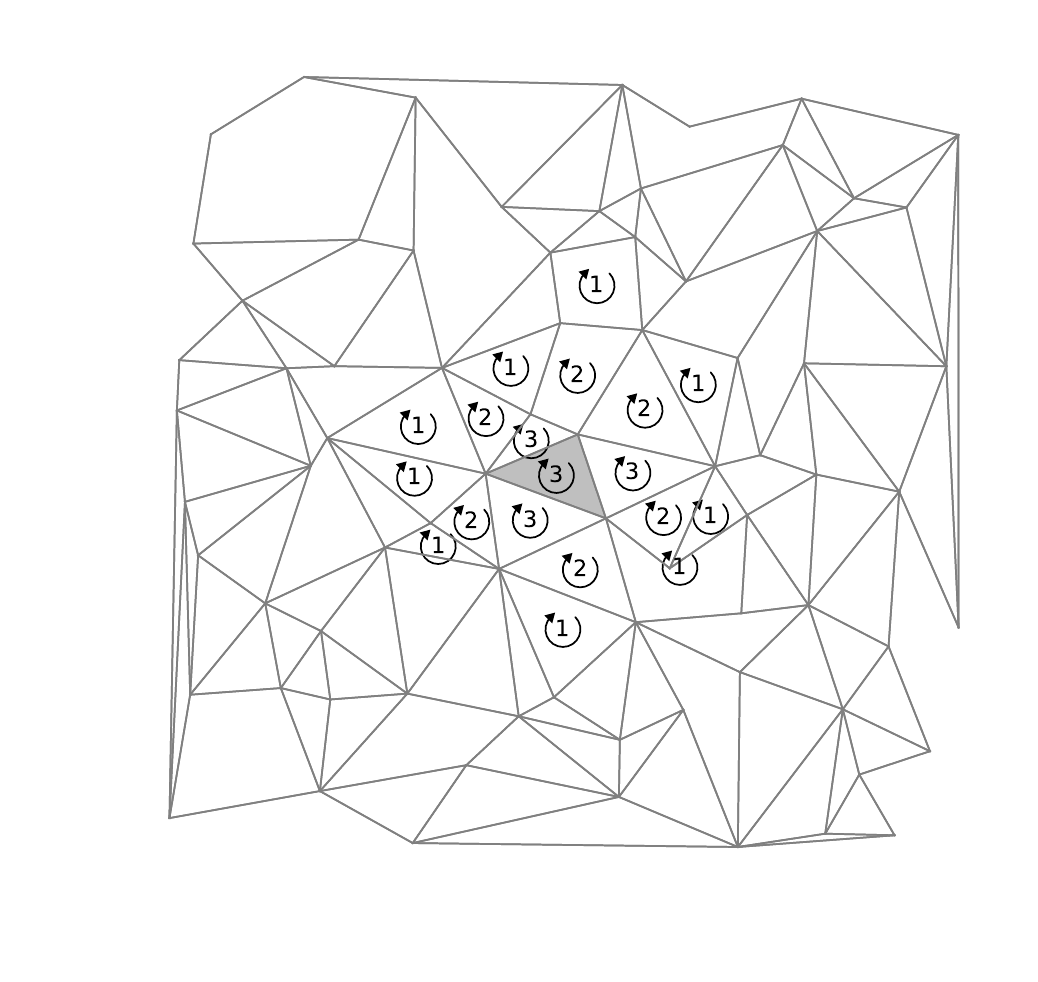}}}  
\caption{The ``pulse'' (of height 3) on a planar graph with a hole.
  The top shows the edge representation and the bottom shows the face representation of the initial and final configurations.}
\label{fig:planar}
\end{figure}

The results described above also hold for \emph{finite}
planar graphs.  In this case the external face is included in the underlying complex.

\subsection{Higher dimensional complexes} 

The flow-firing process on the grid (or a planar graph) is a form of
two-dimensional chip-firing.

More generally, one can work over the $n$-dimensional grid (or a
polytopal decomposition of $n$-dimensional space) and define a
ridge-firing process.  A ridge configuration is an integer assignment
to the $(n-1)$-dimensional faces of an $n$-dimensional complex.  The
ridge-firing process ``reroutes'' the value of a ridge to neighboring
ridges along common facets.

Conservative configurations will, by definition, afford facet
representations.  The conservation requirement is a natural topological
condition.  To be conservative, the ridge configuration must be in the
image of a boundary operator on facets.  In particular, the boundary
of a single facet (edges of a square, squares of a cube, etc.) is a
conservative ridge configuration.  

The same boundary operator is used to define the combinatorial
Laplacian for the complex.  The combinatorial Laplacian in turn
dictates the rerouting rules of higher-dimensional chip-firing.  For
the $n$-dimensional grid (or polytopal decomposition), every ridge is
contained in exactly two facets and the ridge-firing process in terms
of the face representation is precisely the same as in the
$2$-dimensional (flow-firing) case.  If two neighboring facets differ
by $2$ or more units of flow, they can fire to balance out.

Theorem~\ref{thm:conservative} follows unchanged: If the initial state
is a finite conservative ridge configuration then the ridge-firing
process terminates after a finite number of steps.

Theorem~\ref{thm:main} also follows unchanged: If the initial state is
a conservative ridge configuration around a topological obstruction
$\sigma$ then the ridge-firing process terminates in a finite
number of steps at a \emph{unique} final configuration with the
prescribed face representation.

\bibliographystyle{amsalpha}
\bibliography{biblio.bib}

\end{document}